\documentclass[11pt,reqno]{amsart} 

\usepackage{amsrefs}
\usepackage{epsf,latexsym,graphicx}
\usepackage[matrix,arrow,tips,curve]{xy}
\usepackage[leqno]{amsmath}
\usepackage[usenames,dvipsnames,svgnames,table]{xcolor}
\usepackage{enumitem}
\usepackage{mathtools}
\usepackage{amsfonts}
\usepackage{hyperref}
\usepackage{blkarray}
\usepackage{nicefrac}
\usepackage{amssymb}
\hypersetup{
    colorlinks,
    linkcolor={red!50!black},
    citecolor={blue!50!black},
    urlcolor={blue!80!black}
}

\theoremstyle{plain}
\newtheorem{theorem}{Theorem}[section]
\newtheorem{corollary}[theorem]{Corollary}
\newtheorem{lemma}[theorem]{Lemma}
\newtheorem{proposition}[theorem]{Proposition}
\theoremstyle{definition}
\newtheorem{definition}[theorem]{Definition}
\newtheorem{remark}[theorem]{Remark}

\providecommand{\Z}{{\ensuremath{\mathbb{Z}}}}

\providecommand{\C}{{\ensuremath{\mathbb{C}}}}
\providecommand{\HH}{{\ensuremath{\mathrm{H}}}}
\providecommand{\pr}{{\ensuremath{\mathbb{P}}}}
\providecommand{\calX}{{\ensuremath{\mathcal{X}}}}
\providecommand{\calU}{{\ensuremath{\mathcal{U}}}}

\providecommand{\I}{{\ensuremath{\mathnormal{I}}}}
\providecommand{\pp}{{\ensuremath{\pr^1\times\pr^1}}}
\providecommand{\Hilb}{\ensuremath\mathrm{Hilb}}
\providecommand{\spn}{\ensuremath\mathrm{span}}

\providecommand{\rank}{\ensuremath{\mathrm{rank}}}

\providecommand{\VPS}{\ensuremath{\mathrm{VPS}_{\pp}}}
\providecommand{\VPowerS}{\ensuremath{\mathrm{VPS}}}
\providecommand{\VSumP}{\ensuremath{\mathrm{VSP}}}
\providecommand{\VPSS}{\ensuremath{\mathrm{VPS}_{\Sigma}}}
\providecommand{\VPSC}{\ensuremath{\mathrm{VPS}_{C}}}
\providecommand{\tth}{\thinspace}
\providecommand{\Pic}{\ensuremath{\mathrm{Pic}}}
\providecommand{\SD}{\ensuremath{SD}}

\begin{document}
\begin{center}
\title[Varieties of apolar subschemes of toric surfaces]
{Varieties of apolar subschemes of toric surfaces}
\par \bigskip
\author[M.~Gallet]{Matteo Gallet}
\address{(Matteo Gallet) RICAM, Austrian Academy of Sciences, Altenberger 
Stra{\ss}e 69, A-4040, Linz, Austria.}
\email{matteo.gallet@ricam.oeaw.ac.at}
\author[K.~Ranestad]{Kristian Ranestad}
\address{(Kristian Ranestad) University of Oslo, Department of Mathematics, PO 
Box 1053, Blindern, N-0316 Oslo, Norway}
\email{ranestad@math.uio.no}
\author[N.~Villamizar]{Nelly Villamizar}
\address{(Nelly Villamizar) Swansea University, Department of Mathematics, Singleton Park, SA2 8PP, Swansea, UK.}
\email{n.y.villamizar@swansea.ac.uk}

\date{\today}
\maketitle
\end{center} 	

\begin{abstract}
Powersum varieties, also called varieties of sums of powers, have provided 
examples of interesting relations between varieties since their first 
appearance in the 19th century. One of the most useful tools to study them 
is apolarity, a notion originally related to the action of differential 
operators on the polynomial ring. In this work we make explicit how one can 
see apolarity in terms of the Cox ring of a variety. In this way powersum 
varieties are a special case of varieties of apolar schemes; we 
explicitly describe examples of such varieties in the case of two toric 
surfaces, when the Cox ring is particularly well-behaved.
\end{abstract}

\section{Introduction}\label{intro}

Powersum varieties, parametrizing expressions of a form as a sum of powers of 
linear polynomials, provide examples of surprising relations between varieties, 
namely between them and the hypersurfaces defined by the forms. These varieties 
have been widely studied since the 19th century: Sylvester considered and solved 
the case of binary forms (see \cites{Sylvester1904a, Sylvester1904b}). A number 
of further cases have been treated more recently, see \cite {IR}, \cite{MM}, 
\cite{Muk92}, \cite{RS1}, \cite{RS2} and~\cite{RV}.

A powersum variety~$\VSumP(f,k)$ is associated to a polynomial 
$f\in\operatorname{Sym}^{d} V$ on a vector space~$V$ over a field $K$, and to a 
positive integer number $k$. It is defined as the Zariski closure in the Hilbert 
scheme of subschemes of~$\pr \left( V \right)$ of length~$k$ of the set
\begin{multline*}
  \Bigl\{ \big[ [l_1], \dotsc, [l_k] \big] \in \Hilb_k \, \pr \left( 
 V \right) \, : \, f = l_1^d + \dotsb 
+ l_k^d, \\  \text{where } [l_i] \in \pr(V) \text{ are pairwise 
distinct} \Bigr\}.
\end{multline*}
Although $K$ in this definition can be an arbitrary field, throughout this paper 
we always consider it to be the field of complex numbers $\C$.

Powersum varieties are special cases of a more general construction:  given a 
projective variety $X\subseteq \pr^n$, let $y\in \pr^n$ be a general point and 
$k$ be the minimal integer such that there are  $k$ points in~$X$ whose 
span contains~$y$. Thus $k$ is the smallest integer such that the $(k-1)$-th 
secant variety of~$X$ fills the space~$\pr^n$. Define $\VPowerS_X(y,k)$ to be 
the closure in $\Hilb_k(X)$ of the set of smooth subschemes of length~$k$ whose 
span contains~$y$. If $X$ is a Veronese variety, then one may 
interpret~$y$ as the class~$[f]$ of a homogeneous form~$f$ of
some degree~$d$. The 
$k$-tuples of points on~$X$ whose span contains $[f]$ represent expressions 
of~$f$ as a sum of $d$-th powers of linear forms, hence in this way we recover 
the notion of powersum variety of~$f$ and~$k$. Therefore, if $X \subseteq \pr 
\bigl( \operatorname{Sym}^{d} V \bigr)$ is the $d$-uple embedding of~$\pr(V)$, 
then $\VSumP(f,k) = \VPowerS_X([f],k)$.

Furthermore, $\VPowerS_X(y,k)$ can be seen as the ``Variety of aPolar 
Subschemes":

\begin{definition}\label{def:apolar}
A subscheme $Z \subseteq X$ is called apolar to $y\in \pr^n$, if $y$ is 
contained in the linear span of $Z$ in  $\pr^n$.
\end{definition}

Note that $\VPowerS_X(y,k)$ contains only those apolar subschemes that are in 
the closure of the set of smooth apolar subschemes. When $\dim X>3$, then the 
general point~$y$ may have singular apolar subschemes of length~$k$ that do not 
belong to this closure.
 
Apolar schemes have been studied, in the classical setting of powersum 
varieties, considering the ideal of differential operators annihilating a given 
homogeneous form. More precisely, one considers two polynomial rings $S = 
\C[x_0, \dotsc, x_n]$ and $T = \C[y_0, \dotsc, y_n]$, and the action of the 
variables $y_i$ on $x_j$ defined by differentiation: $y_i \cdot x_j = \partial 
x_j / \partial x_i$. In this way, for a homogeneous polynomial~$f$ of 
degree~$d$, one defines the set $H_{f} \subseteq T$ of differential operators of 
degree~$d$ annihilating~$f$. Then $H_{f}$ is a hyperplane in the vector space 
$\C[y_0, \dotsc, y_n]_d$ and a scheme $Z$ is apolar to~$[f]$ if and only if 
$I_{Z,d} \subseteq H_f$. 

In this paper we generalize the notion of apolarity and investigate 
$\VPowerS_{X}(y,k)$ using the Cox ring~$\operatorname{Cox}(X)$ of~$X$, namely 
the $\C$-algebra of sections of all line bundles on~$X$, graded by the Picard 
group~$\Pic(X)$. If $X$ is a toric variety, $\operatorname{Cox}(X)$ has a simple 
structure, namely it is a polynomial ring (see~\cite{Cox}): we use this fact to 
explore $\VPowerS_X$ in particular cases of toric surfaces.

Let $T = \bigoplus_{A\in\Pic(X)} T_A$ be $\operatorname{Cox}(X)$, 
and for each $A \in \Pic(X)$ let $S_A$ be the space of linear forms 
on~$T_A$, i.e.\ $S_A = T_A^{*}$. If $A$ is very ample on $X$, then its global 
sections~$T_A$ define the embedding $\nu_A \colon X \hookrightarrow \pr(S_A)$. 
To every subscheme $Z \subseteq X$ we associate
an ideal $I_Z \subseteq T$:
\[
  I_Z := \bigoplus_{B \in \Pic(X)} I_{Z,B} \subseteq T ;\qquad I_{Z,B}:=\{g\in 
T_B \, \colon \, g|_{Z}\equiv 0\}.
\]
Definition~\ref{def:apolar} may be generalized as follows:

\begin{definition}
\label{def:hyperplane}
For any  nonzero $f \in S_A$ let $H_f \subseteq T_A$ be the hyperplane of 
sections that vanish at the point $[f]\in \pr(S_A)$.
A subscheme $Z \subseteq X$ is called apolar to $f\in S_A$ if 
$I_{Z,A} \subseteq H_f$.
\end{definition}

One of the facts that makes the classical theory of apolarity such a powerful 
tool is that it makes possible to translate the previous condition of 
containment of vector spaces into a condition involving ideals. This fact can be 
generalized as follows. For each $B \in \Pic(X)$ define
\[
I_{f,B}=
\begin{cases}
 H_f:T_{A-B}=\{g\in T_{B} \, \colon \, g\cdot T_{A-B}\subseteq H_f\}, 
\quad\text{if } A-B>0\\
 T_B, \quad \text{otherwise},
\end{cases}
\]
where $A-B>0$ if the line bundle $A-B$ has global sections,
and set
\begin{equation}
\label{eq:def_orthogonal}
  I_f=\bigoplus_{B \in \Pic(X)} I_{f,B} \subseteq T.
\end{equation}
The classical apolarity lemma (see~\cite{IK}*{Lemma~1.15}) can be read as 
follows.
\begin{lemma}\label{apolarity}
For a subscheme $Z \subseteq X$ and $f \in S_A$, then $I_Z\subseteq I_f$ if 
and only if $I_{Z,A} \subseteq H_f$.
\end{lemma}
\begin{proof}  It suffices to show the if-direction of the equivalence.
 If $B>A$ then $I_{f,B}=T_B$, so it suffices to consider $B<A$.  But
$
  I_{Z,B}\cdot T_{A-B} \subseteq I_{Z,A} \subseteq H_f
$
implies
$
  I_{Z,B} \subseteq H_f:T_{A-B}=I_{f,B}
$
and the lemma follows.
\end{proof}
Additionally, we rephrase in terms of the Cox 
ring the maps associating to a polynomial all its partial derivatives of a 
given order: for $A,B\in \Pic(X)$ and $f\in S_A$ we define the linear map 
\begin{equation}\label{apolarmap}
\phi_{f,B} \colon T_B \longrightarrow S_{A-B}; \qquad g\mapsto g(f)  
\end{equation}
such that 
\begin{gather*}
  g(f)(g')=g'g(f)\in \C\quad {\rm for}\quad g'\in T_{A-B} \quad \text{i.e.} \\
  H_{g(f)} := \bigl( H_f:\langle g\rangle \bigr) \subseteq T_{A-B}.
\end{gather*}
Notice that $\operatorname{ker}\,\phi_{f,B}=I_{f,B}$. 

The previous generalization of apolarity was also recently 
considered in~\cite{Galazka} in the particular case of smooth toric 
varieties. The author uses the apolarity lemma to prove upper bounds on the 
minimum length of subschemes whose linear span contains a general point.

\smallskip
In this paper we present three examples where we describe $\VPowerS_X([f],r_f)$ 
where $X$ is a toric surface different from the projective plane, $f$ a general 
section in~$S_A$ for some $A\in \Pic(X)$ and $r_f$ the minimal integer~$r$ such 
that $\VPowerS_X([f],r)$ is not empty.
In Section~\ref{prelim} we set up the theory for apolarity in the case $X = 
\pp$.
We split the proof of the following theorem among Sections~\ref{biquadratic}, 
\ref{bicubic} and \ref{cubicone}.

\begin{theorem}
\label{thm:main}
Let $X$ be a projective variety, $A \in \Pic(X)$ and $f\in S_A$ be a general 
section.
 \begin{enumerate}[label=(\Alph*)]
  \item
  \label{A}
If $X =\pr^1\times \pr^1$ and $A = (2,2)$, then $\VPowerS_\pp([f],4)$ is a 
threefold isomorphic to a smooth linear complex in the Grassmannian $G(2,4)$ 
blown up along a rational normal quartic curve.
  \item
  \label{B}
If $X =\pr^1\times \pr^1$ and $A = (3,3)$, then $\VPowerS_\pp([f],6)$ is 
isomorphic to a smooth Del Pezzo surface of degree~$5$.
  \item
  \label{C}
If $X = F_1$, namely the blow up of~$\pr^2$ in one point embedded as a cubic 
scroll in~$\pr^4$, and $A = 3H$ where $H$ is the hyperplane class of~$F_1$, 
then $\VPowerS_{F_1}([f],8)$ is isomorphic to~$\pr^2$ blown up in $8$ points.
 \end{enumerate}
\end{theorem}
 
Apolar rational or elliptic curves play a crucial role in our arguments, in 
particular in the use of the following facts.
For rational curves Sylvester showed (see~\cite{Sylvester1904a}):
\begin{lemma} 
\label{rat}
Let $C \subseteq \pr^{2d-1}$ be a rational normal curve of degree~$2d-1$, then 
there is a unique $d$-secant $\pr^{d-1}$ to~$C$ passing through a general 
point, i.e.\ $\VPowerS_C(y,d) \cong \{pt\}$ for a general point 
$y\in\pr^{2d-1}$. 
Let $C \subseteq \pr^{2d}$ be a rational normal curve of degree~$2d$ and $y$ a 
general point in~$\pr^{2d-1}$, then $\VPowerS_C(y,d+1)\cong C$.
\end{lemma}

For elliptic curves, the following lemma follows from  Room's description of 
determinantal varieties. We give a proof in Lemma~\ref{ell} in Section~\ref{cubicone}.
\begin{lemma}
\label{elliptic}
Let $C \subseteq \pr^{2d-2}$ be an elliptic normal curve of degree~$2d-1$ and 
$y$ a general point in~$\pr^{2d-1}$, then $\VPowerS_C(y,d)\cong C$.
\end{lemma}

\section{Apolarity for $\pp$}
\label{prelim}

Let us consider $X = \pp$. In this case, the Picard group of $X$ is~$\Z^2$ and 
its Cox ring is $T = \C[t_0,t_1][u_0,u_1]$, see  for instance 
\cite{CLS}*{Example 5.2.2}.

We can write $T = \bigoplus_{a,b\in\Z} T_{a,b}$, 
where $T_{a,b}$ is the set of bihomogeneous polynomials of bidegree~$(a,b)$. In 
this case, setting $S_{a,b} = T_{a,b}^{*}$, the group $S = \bigoplus_{a,b\in\Z} 
S_{a,b}$ has the structure of a ring. In fact, $S = \C[x_0,x_1][y_0,y_1]$ where 
the duality between homogeneous components of~$S$ and $T$ is induced by 
differentiation: $t_i = \partial/\partial x_i$ and $u_i = \partial/\partial 
y_i$. If $f\in S_{a,b}$ and the annihilator is defined as 
\[f_{c,d}^\bot = \{ g 
\in T_{c,d} \, \colon \, g(f) = 0 \},\] 
then $f_{a,b}^\bot =H_f$, where $H_f$ is 
the hyperplane from Definition~\ref{def:hyperplane}. Setting $f_{c,d}^\bot = 
f_{a,b}^\bot \colon T_{(a-c, b-d)}$, the annihilator ideal of $f$ and the ideal 
$I_f$ defined in~\eqref{eq:def_orthogonal} coincide:
\[
  f^\bot := \bigl\{ g \in T \, \colon \, g(f) = 0 \bigr\}=I_f.
\]
When $a,b>0$, the divisors of class $(a,b)$ on~$\pp$  determine
the Segre-Veronese embedding
\begin{equation}
\label{eq:svembedding}
  \begin{array}{rccc}
    \nu_{a,b} \colon & \pp & \hookrightarrow & \pr^{ ab + a + b } \\
    & \big( [l_1], [l_2] \big) & \mapsto & [l_1^a l_2^b],
  \end{array}
\end{equation}
where $l_1 \in \left\langle x_0, x_1 \right\rangle$, $l_2 \in \left\langle 
y _0, y_1 \right\rangle$, and $\pr^{ ab + a + b }$ is identified with~$\pr(S
_{ a,b})$.  We sometimes call $ \nu_{a,b}$ the $(a,b)$-embedding. 

A subscheme $\Gamma\subseteq \pp$ is apolar to $f$ if  $I_{\Gamma, (a,b) } 
\subseteq f^\bot_{a,b}$, or equivalently  $[f] \in \spn \, \nu_{a,b}(\Gamma)$.
The variety of apolar schemes $\VPS([f],r)$ may be interpreted as a 
variety of sums of powers, i.e.\ as the Zariski 
closure of
\[
  \biggl\{ \bigl[ ([l_{11}], [l_{21}]) , \dots, 
  ([l_{1r}], [l_{2r}]) \bigr] \in \Hilb_r \left( \pp \right) \,
  \colon \, f = \sum_{i=1}^r l_{1i}^a \tth l_{2i}^b \biggr\}.
\]
As in the standard homogeneous case the minimal~$r$ such that $\VPS([f], r)$ is 
nonempty, is called {\em the rank of~$f$}, denoted~$\rank(f)$.

Let us remark that a general form in~$S_{a,b}$ has rank~$r$ if and only if $r$ 
is the minimal~$k$ such that the $k$-secant variety of the Segre-Veronese 
coincides with~$\pr^{ab+a+b}$.

The computation of the dimension of secant varieties carried
in~\cite{CGG2005a}*{Corollary 2.3}, implies that if $f$ 
is a bihomogeneous general form of bidegree~$(a,b)$, then
\begin{equation}
\label{eq:rankf}
   \rank(f) \, = \,
  \begin{cases}
  \;\; 2d+2 & \mbox{if } (a,b) = (2,2d) \text{  for some } d,\\
   \biggl\lceil{\frac{ ( a+1 )( b+1 ) }{ 3 }} \biggr\rceil
  & \mbox{otherwise.} 
   \end{cases}
\end{equation}
For such a general form~$f$, the dimension 
of~$\VPS([f],r)$ is determined by $\rank(f)$ as described in the following 
proposition (see \cite{Dolgachev2004}*{Proposition~3.2} for the 
classical case). 

\begin{proposition}
\label{prop:dimension}
Let $f \in S_{a,b}$ be a general bihomogeneous form of rank~$r$. Then 
$\VPS([f],r)$ is an irreducible variety of dimension
\begin{equation*}
  \dim \VPS([f],r) \, = \,
  \begin{cases}
  \hskip 2cm 3 & \mbox{if } (a,b) = (2,2d),\\
  3\biggl(\biggl\lceil  {\frac{ ( a+1 )( b+1 ) }{ 3 }}\biggr\rceil -
  \frac{ ( a+1 )( b+1 ) }{ 3 }\biggr)&\mbox{if } (a,b)\neq (2,2d).  
  \end{cases}
\end{equation*}
\end{proposition}

\begin{proof}
Let us denote $\Hilb_r \left( \pp \right)$ 
by~$\mathcal{H}$. We consider the incidence variety
\[ 
  \mathcal{X} = \Big\{ \bigl([\Gamma], [f] \bigr) \in \mathcal{H} \times 
  \pr^{ab + a +b} \, \colon \, 
  [\Gamma] \in \VPS([f],r) \Big\}\,.
\]
Then we have the two projection maps:
\[ 
  \pi_1 \colon \calX \longrightarrow \mathcal{H} \qquad \qquad 
  \pi_2 \colon \calX \longrightarrow \pr^{ab+a+b} 
\]
Let $\calU$ be the open subset of~$\mathcal{H}$ parametrizing zero-dimensional
schemes given by~$r$ distinct points in~$\pp$.
It is possible to restrict~$\calU$ so that the $(a,b)$-th powers of 
all linear forms associated to such points are linearly independent. 
In this way we can prove that $\pi_1$ is dominant. Moreover, if $[\Gamma] \in 
\calU$, the fiber of~$\pi_1$ over~$[\Gamma]$ is an open set of a linear 
space of dimension~$r-1$. Since $\mathcal{H}$ is irreducible 
\cite{Fogarty1968}, then also~$\mathcal{X}$ is irreducible 
and of dimension~$3r-1$. The fiber of~$\pi_2$ over $[f] \in \pr^{ab +a +b}$ 
is~$\VPS([f],r)$, so for a general~$f$, the variety $\VPS([f],r)$ 
has dimension $3r-1-(ab+a+b)$. Using formula~\eqref{eq:rankf} the statement 
follows.
\end{proof}

\section{Bihomogeneous forms of bidegree $(2,2)$}\label{biquadratic}

The Segre-Veronese embedding~\eqref{eq:svembedding} is in this 
case the $(2,2)$-embedding of~$\pr^1\times \pr^1$ in~$\pr^8$, denoted 
$\nu_{2,2} \colon \pr^1 \times \pr^1 \hookrightarrow \pr^8$.
If $f$ is a general bihomogeneous form in~$S_{2,2}$, then by 
the formula of the rank~\eqref{eq:rankf} applied to the case $(a,b) = (2,2d)$ 
with~$d=1$, we have~$\rank(f)=4$, and $\dim \VPS ([f], 4) = 3$
by Proposition~\ref{prop:dimension}.

\begin{lemma}
\label{lem:orthogonal_biquadratic}
For a general form $f\in S_{2,2}$ the orthogonal~$f^\bot$ is generated 
by~$f_{2,1}^\bot, f_{1,2}^\bot, f_{3,0}^\bot$ and 
$f_{0,3}^\bot$. Moreover, both $f^{\bot}_{2,1}$ and $f^{\bot}_{1,2}$ have 
dimension~$4$.
\end{lemma}
\begin{proof} 
Since $\dim T_{2,1} = 6$ and $\dim S_{0,1} = 2$, and $f$ is general, the 
$\ker \phi_{f,(2,1)} = f_{2,1}^\bot$ has dimension~$4$, with $\phi_{f,(2,1)}$ as 
defined in~\eqref{apolarmap}. By symmetry, also $f^\bot_{1,2}$ has 
dimension~$4$.

Consider the vector subspace $T_{0,1} \cdot f^{\bot}_{2,1} \subseteq 
f^{\bot}_{2,2}$: if it is of dimension~$8$, it means that we do not need 
elements from~$f^{\bot}_{2,2}$ to generate $f^{\bot}$. Suppose that $\dim 
T_{0,1} \cdot f^{\bot}_{2,1} < 8$: if $g_1, \dotsc, g_4$ is a basis 
for~$f^{\bot}_{2,1}$, then $u_0 \tth g_1, \dotsc, u_0 \tth g_4$, $u_1 \tth g_1, 
\dotsc, u_1 \tth g_4$ are linearly dependent. So, for some $h_1, h_2 \in 
f^{\bot}_{2,1}$ we have $u_0 \tth h_1 + u_1 \tth h_2 = 0$. Then $h_1 = u_1 \tth 
\tilde{h}$ and $-h_2 = u_0 \tilde{h}$ for some nonzero $\tilde{h} \in T_{0,2}$. 
Hence $T_{0,1} \cdot \tilde{h} \subseteq f^{\bot}$, which forces $\tilde{h} \in 
f^{\bot}$. But $\tilde{h}$ has bidegree~$(2,0)$, and by the generality 
assumption on~$f$ there is no nontrivial element in~$f^{\bot}_{2,0}$. Hence 
$\dim T_{0,1} \cdot f^{\bot}_{2,1} = 8$, and so $T_{0,1} \cdot f^{\bot}_{2,1} = 
f^{\bot}_{2,2}$.

Moreover, since $f$ is a form of bidegree $(2,2)$, then $f^{\bot}_{a,b} = 
T_{a,b}$ whenever~$a$ or~$b$ is greater than or equal to~$3$. Notice that 
$T_{3,b} = T_{3,0}\cdot T_{0,b}$, and $T_{a,3} = T_{0,3} \cdot T _{a,0}$, for 
every $a, b\geq 1$. Thus, $f^\bot_{2,1}$, $f^\bot_{1,2}$ together with 
$f^\bot_{3,0} = T_{3,0}$ and $f^\bot_{0,3} = T_{0,3}$ generate the 
ideal~$f^\bot$. 
\end{proof}

The space of sections $f^{\bot}_{2,1}$ defines a linear system of 
$(2,1)$-curves 
on~$\pp$ and, by Lemma~\ref{lem:orthogonal_biquadratic}, a rational map 
$\delta_{2,1} \colon \pp \longrightarrow \pr^3$ fitting in the diagram
\begin{equation}
\label{eq:delta}
  \begin{array}{c}
  \xymatrix@R=.0cm@C=.6cm{& & \pr^5 \ar@{-->}[dd]^-{\pi_{2,1}} \ar@{}[r]|-{=} & 
\pr\left( S_{2,1} \right) \\ 
\pr^1 \times \pr^1 \ar[rru]^-{\eta_{2,1}} \ar@{-->}[rrd]_{\delta_{2,1}} \\ 
& & \pr^3  \ar@{}[r]|-{=} & \pr \bigl( \left( f^{\bot}_{2,1} \right)^{*} 
\bigr)}
  \end{array}
\end{equation}
where $\eta_{2,1}$ is the map induced by the complete linear systems of 
$(2,1)$-curves in~$\pp$, and $\pi_{2,1}$ is a linear projection. 
%Notice that 
All constructions and results from now on apply to both~$f^{\bot}_{2,1}$ 
and~$f^{\bot}_{1,2}$.

First, we prove that $\delta_{2,1}$ is a morphism. For this we analyze the 
projection center $L_{2,1}$ of $\pi_{2,1}$. By definition, $L_{2,1} \subseteq 
\pr(S_{2,1})$ is spanned by the forms annihilated by $f_{2,1}^{\bot}$, i.e.\ by 
the partials $\partial f / \partial y_0$ and $\partial f / \partial y_1$. 
Consider the surface scroll $Y_{2,1} := \eta_{2,1} \left( \pp \right) \subseteq 
\pr \left( S_{2,1} \right)$. The $(1,0)$-curves on~$\pp$ are mapped to lines 
in~$Y_{2,1}$, while the $(0,1)$-curves are mapped to conics. The planes of these 
conics are the planes spanned by forms $q(x_0,x_1) \cdot l(y_0, y_1)\in 
S_{2,1},$ where $l(y_0, y_1)\in \left\langle y_0,y_1 \right\rangle$ is a fixed 
linear form. We let $W_{2,1}$ be the threefold union of these planes.

\begin{lemma}
\label{lem:partials_22}
Let $f \in S_{2,2}$ be a general form. Then no linear combination of its
partial derivatives $\partial f / \partial y_0$ and $\partial f / \partial y_1$
is of the form $q(x_0,x_1) \cdot l(y_0, y_1)$, where $q$ and $l$ are
respectively a quadratic and a linear form. In particular, the line $L_{2,1} 
\subseteq \pr(S_{2,1})$ does not intersect~$W_{2,1}$.
\end{lemma}
\begin{proof}
Write $f$ as 
$
  y_0^2 \tth Q + y_0 y_1 \tth Q' + y_1^2 \tth Q''.
$
Then, a linear combination $\lambda \tth \nicefrac{\partial f}{\partial y_0} + 
\mu \tth \nicefrac{\partial f}{\partial y_1}$ is of the form~$q \cdot l$ if 
and only if $\lambda Q + \mu Q'$ is proportional to $\lambda Q' + \mu Q''$. 
Since $f$ is general, we can suppose that $Q,\, Q'$ and $Q''$ are linearly 
independent. Then the two pencil of quadrics $\lambda Q + \mu Q'$ and $\lambda 
Q' + \mu Q''$ have at most one point in common, which hence must coincide 
with~$Q'$. But
\[
 \lambda Q + \mu Q' = Q' \;\; \Leftrightarrow \;\; \lambda = 0 
 \qquad \text{and} \qquad
 \lambda Q' + \mu Q'' = Q' \;\; \Leftrightarrow \;\; \mu = 0.
\]
This proves the claim.
\end{proof}

\begin{corollary}
\label{cor:morphism} 
  The map $\delta_{2,1}$ defined by $f^{\bot}_{2,1}$ is a morphism. Moreover, 
all lines in $Z_{2,1}:=\delta_{2,1} \left( \pp   \right)$ are linear projections 
of lines in~$Y_{2,1}=\eta_{2,1} \left( \pp \right)$. In particular, no conic 
in~$Y_{2,1}$ is mapped to a line by~$\pi_{2,1}$. The analogous result holds for 
the $(1,2)$ case.
\end{corollary}
\begin{proof}
Since the projection center~$L_{2,1}$ does not intersect $W_{2,1} \supseteq 
Y_{2,1}$, the projection~$\pi_{2,1}$ restricted to~$Y_{2,1}$ and hence 
$\delta_{2,1}$ are morphisms, and no conic in~$Y_{2,1}$ is mapped to a line 
in~$Z_{2,1}$.
Furthermore, $f_{1,1}^{\bot} = \{0\}$, so $L_{2,1}$ does not lie in the span of 
any $(1,1)$-curve in $Y_{2,1}$, therefore  every line in~$Z_{2,1}$ is the 
linear projection of a $(1,0)$-curve, i.e.\ a line in~$Y_{2,1}$.
\end{proof}

\begin{lemma}
\label{lem:pencils_21}
  Let $f \in S_{2,2}$ be a general form and suppose that $g_1, g_2 \in
f^{\bot}_{2,1}$ span a linear space of dimension~$2$. 
Then either the ideal $(g_1, g_2)$ defines a scheme of length~$4$ in~$\pp$, 
or the pencil of $(2,1)$-curves defined by~$g_1$ and~$g_2$ has a common 
component, a $(1,0)$-curve. 
\end{lemma}
\begin{proof}
  We need to exclude that the pencil of $(2,1)$-curves defined by~$g_1$ 
and~$g_2$ has a fixed component that is a $(2,0)$, a $(1,1)$ or a $(0,1)$-curve. 
We treat these cases one by one.

If there is a common $(2,0)$-curve, then all curves in the pencil split into it 
and a $(0,1)$-line. Therefore there exists $q \in \C[u_0, u_1]_2$ such that
\begin{equation*}
  (\alpha \tth t_0 + \beta \tth t_1) \tth q \cdot f \; = \; 0 \quad \forall 
  \alpha, \beta \in \C,
\end{equation*}
so $q \in f^{\bot}_{0,2}$, but this contradicts 
Lemma~\ref{lem:orthogonal_biquadratic}. Similarly, if we assume that there is a 
common $(1,1)$-curve, then $f^{\bot}_{1,1}$ would be non-trivial and this again
contradicts Lemma~\ref{lem:orthogonal_biquadratic}. 

We are left with the case when the pencil $\left\langle g_1, g_2 
\right\rangle$ has a $(0,1)$-curve $\ell$ in its base locus. Consider the 
maps in the diagram~(\ref{eq:delta}), with $L_{2,1}$ the center 
of the projection~$\pi_{2,1}$. Then $\ell$ is sent to a conic by~$\eta_{2,1}$. 
Hence there is a pencil of hyperplanes in~$\pr^5$ passing through~$L_{2,1}$ 
and having a conic in its base locus. Since the base locus of a pencil 
of hyperplanes in~$\pr^5$ is a~$\pr^3$, both the conic 
and the line~$L_{2,1}$ lie in a~$\pr^3$. The latter happens if and only if 
$L_{2,1}$ intersects the plane spanned by such a conic, but this is not 
possible by Corollary~\ref{cor:morphism}.
\end{proof}

The image $Z_{2,1}$ of~$\pp$ under~$\delta_{2,1}$ is a rational scroll, 
since it is rational and covered by the images of the lines in the scroll 
$\eta_{2,1} \left( \pp \right)$. We describe its singular 
locus.

\begin{lemma}\label{lem:singular_locus}
$Z_{2,1}$ is a quartic surface 
in~$\pr^3$ that has double points along a twisted cubic curve and no triple 
points. 
\end{lemma}
\begin{proof}
A scheme~$z$ of length~$3$ in~$\pp$ is contained in a $(1,1)$-curve $C$, so if 
$z$ is mapped to a point by $\delta_{2,1}=\pi_{2,1}\circ \eta_{2,1}$ the 
projection center~$L_{2,1}$ is contained in the span of $\eta_{2,1}(z) \subseteq 
Y_{2,1}$, and hence in the span of the image of~$C$ in~$Y_{2,1}$. But then $C$ 
is mapped to a line in~$Z_{2,1}$, and this is excluded by 
Corollary~\ref{cor:morphism}. Therefore $Z_{2,1}$ has no triple points. A 
general plane section of~$Z_{2,1}$ is the image of a smooth rational quartic 
curve in~$Y_{2,1}$; therefore it is a rational quartic curve, and so has $3$ 
singular points that span a plane. Hence $\operatorname{Sing}\left( Z_{2,1} 
\right)$ spans~$\pr^3$ and is a cubic curve. From the double point formula 
(see~\cite{fulton1998}*{Theorem~9.3}) we see that the double point locus of the 
restriction~$\pi_{2,1}|_{{Y_{2,1}}}$ is a curve on~$Y_{2,1}$ of degree~$6$ 
because it is linearly equivalent to $-K_{Y_{2,1}}$ (the anticanonical divisor 
of~$Y_{2,1}$). If~$\,\operatorname{Sing}\left( Z_{2,1} \right)$ is not a twisted 
cubic, then it has to contain a line. The preimage under~$\pi_{2,1}$ of such a 
line is then either a conic~$C$ or two skew lines~$E_1$ and $E_2$. In the first 
case the center~$L_{2,1}$ of~$\pi_{2,1}$ intersects~$\spn(C)$, but this is not 
possible because of Lemma~\ref{lem:partials_22}. In the second case we have 
$L_{2,1} \subseteq \spn (E_1 \cup E_2)$. Since $\spn (E_1 \cup E_2) \cong 
\pr^3$, there is a pencil of hyperplanes in~$\pr^5$ containing it; each of them 
intersects~$Y_{2,1}$ in~$E_1 \cup E_2$ and in a residual conic~$D$. Hence we 
obtain a pencil of conics~$D$ such that $\spn(D)$ intersects~$\spn (E_1 \cup 
E_2)$ in a line~$F$. In this way we get a pencil of lines~$F$ in~$\spn (E_1 \cup 
E_2)$; such pencil fills a quadric in~$\spn (E_1 \cup E_2)$, and therefore the 
center~$L_{2,1}$ intersects this quadric in $2$ points; this situation is again 
ruled out by Lemma~\ref{lem:partials_22}. Therefore the only possibility left is 
that $\operatorname{Sing}\left( Z_{2,1} \right)$ is a twisted cubic.
\end{proof}

\begin{remark}
\label{remark:collinear}
Consider a smooth scheme $[\Gamma] \in \VPS([f],4)$ apolar to~$f$, namely 
$\I_{\Gamma}\subseteq f^\bot$. Notice that the dimension of~$I_{\Gamma, (2,1)}$ 
equals the number of linearly independent planes in~$\pr^3$ passing 
through~$\delta_{2,1}(\Gamma)$. Moreover $2 \leq \dim I_{\Gamma, (2,1)} \leq 3$, 
where the latter inequality follows since $\delta_{2,1}$ is defined on the 
whole~$\pp$. Lemma~\ref{lem:singular_locus} excludes that the dimension 
of~$I_{\Gamma, (2,1)}$ is~$3$, since in that case we would have 
$\delta_{2,1}(\Gamma) = \{ \mathrm{pt} \}$. Hence $\dim 
I_{\Gamma, (2,1)} = 2$ and $\delta_{2,1}(\Gamma)$ spans a line.
\end{remark}

Remark~\ref{remark:collinear} yields a rational map defined on 
smooth apolar schemes:
\begin{equation}
\label{eq:birational_22}
  \begin{array}{rccc}
    \Phi_{2,1} \colon & \VPS([f],4) & \dashrightarrow & G \left( 2,f_{2,1}^\bot 
\right) \\
    & [\Gamma] & \mapsto & \I_{\Gamma, (2,1)}
  \end{array}
\end{equation}

\begin{lemma}\label{lem:map22}
For a general bihomogeneous form~$f$ of bidegree~$(2,2)$, the rational 
map~$\Phi_{2,1}$ in~\eqref{eq:birational_22} extends to a morphism on 
the whole $\VPS([f],4)$. Let $D_{2,1}$ be the curve 
\begin{equation}
\label{eq:D}
 D_{2,1} \, = \, \left\{ [\ell] \in G \left( 2,f_{2,1}^\bot \right) \, \colon 
\, 
\ell \subseteq Z_{2,1} \right\} \quad \text{where} \quad Z_{2,1} = \delta_{2,1} 
\left( \pp \right).
\end{equation}
Then the fiber over a 
point~$p$ under~$\Phi_{2,1}$ is a smooth rational curve if $p \in D_{2,1}$ and 
it is at most one point when $p \not\in D_{2,1}$. In particular, $\Phi_{2,1}$ 
is birational.
\end{lemma}
\begin{proof}
Since being collinear (see Remark \ref{remark:collinear}) is a closed 
property, $\Phi_{2,1}$ extends to the closure of smooth apolar schemes, namely 
to~$\VPS([f],4)$. 

Let $[\Gamma] \in \VPS([f],4)$ and let $\ell_{\Gamma}$ be the line in~$\pr^3$ 
containing $\delta_{2,1}(\Gamma)$. If $[\ell_{\Gamma}] \not\in D_{2,1}$, then 
$\ell_{\Gamma} \cap Z_{2,1}$ is a scheme of length~$4$, namely it 
is~$\delta_{2,1}(\Gamma)$. Hence the fiber over $\ell_{\Gamma}$ is 
exactly~$[\Gamma]$. Therefore, if $\ell \subseteq \pr^3$ is any line not 
contained in~$Z_{2,1}$, then either $[\ell]$ is not in the image 
of~$\Phi_{2,1}$, or it is the image of exactly one scheme in~$\VPS([f],4)$.

Let $\ell \subseteq \pr^3$ be a line contained in~$Z_{2,1}$. Denote by~$\ell'$ 
the $(1,0)$-line in~$\pp$ such that $\delta_{2,1}(\ell') = \ell$. Since the 
preimage of~$\,\operatorname{Sing}(Z_{2,1})$ under the projection is linearly 
equivalent to the anticanonical divisor of~$Y$, and so every line in~$Y_{2,1}$ 
intersects it in two points, then $\ell$ intersects 
$\operatorname{Sing}(Z_{2,1})$ in $2$ points. In fact, by 
Corollary~\ref{cor:morphism}, every line in~$Z_{2,1}$ is a projection of a line 
in~$Y$. We know that the preimage of $\ell \cap \mathrm{Sing}(Z_{2,1})$ 
under~$\delta_{2,1}$ consists of a scheme of length $4$ that intersects $\ell'$ 
in a subscheme of length~$2$. Summing up, $\delta_{2,1}^{-1} (\ell) = \ell' 
\cup \{z_{\ell}\}$, where $z_{\ell}$ is a scheme of length~$2$ that is mapped 
to~$\ell$ by~$\delta_{2,1}$.

Let $\Gamma \subseteq \pp$ be an apolar scheme such that $\delta_{2,1}(\Gamma)$ 
is contained in~$\ell$. Since $\Gamma$ has length~$4$ and $\Gamma \subseteq 
\delta_{2,1}^{-1} (\ell)$, then the span of $\eta_{2,1}(\Gamma)$ must 
contain~$\ell'$. On the other hand, since $\Gamma$ is apolar to~$f$, it is also 
apolar to both $\partial f / \partial y_0$ and $\partial f / \partial y_1$. 
Therefore the span of $\eta_{2,1}(\Gamma)$ contains the line~$L_{2,1}$, the 
center of the projection~$\pi_{2,1}$ in diagram~\eqref{eq:delta}. This rules 
out the case $\spn \, \Gamma \cong \pr^1$ and $\spn \, \Gamma \cong 
\pr^2$, since in both cases the center of the projection would 
intersect~$Y_{2,1}$, contradicting Corollary~\ref{cor:morphism}. Hence 
$\eta_{2,1}(\Gamma)$ spans a~$\pr^ 3$ in $\pr^5$, and therefore there is a 
length~$2$ subscheme of~$\Gamma$ that is not contained in $\ell'$. Clearly, 
this subscheme coincides with~$z_{\ell}$. 

Let us consider a plane in~$\pr^3$ through~$\ell$. By construction of the 
map~$\ell$, such plane defines (up to scalars) a form $g \in f^{\bot}_{2,1}$ 
that factors as $g = l \tth \tilde{g}$, where $l$ is a $(1,0)$-form whose 
vanishing locus in~$\pp$ is~$\ell'$. If $g$ vanishes on a scheme $\Gamma$ of 
length~$4$, that is apolar to $f$ and is mapped to $\ell$ by $\delta_{2,1}$, 
then $z_{\ell} \subseteq \Gamma$ and $\tilde{g}$ must vanish on the length~$2$ 
subscheme~$z_{\ell}$. In fact there is a pencil of $(1,1)$-forms vanishing 
on~$z_{\ell}$ that together with~$l$ vanish on~$\Gamma$. Therefore every 
subscheme in the fiber over a $\Phi_{2,1}(\Gamma)$ contains~$z_{\ell}$ and is 
contained in the reducible curve $\ell' \cup Z(\tilde{g})$.

The set of apolar subschemes of length~$4$ in $\ell'\cup Z(\tilde{g})$ are 
described in the following lemma.

\begin{lemma}\label{lem:2pencils}
Let $g\in f_{2,1}^\bot$ such that $g = l\tth\tilde{g}$ for forms~$l$ 
and~$\tilde{g}$ of bidegree~$(1,0)$ and~$(1,1)$ respectively. Then the zero 
locus~$C_g$ of~$g$ supports two pencils of length~$4$ apolar schemes. One pencil 
has a common subscheme of length $2$ on $Z(\tilde{g})$ and a moving subscheme of 
length~$2$ on~$Z(l)$, while the subschemes of the other pencil has a unique )
common point on~$Z(l)$, and a moving subscheme of length $3$ on $Z(\tilde{g})$.
\end{lemma}
\begin{proof}
The fact that $g$ is apolar to~$f$ means that the point~$[f]$ is contained in 
the span of the $(2,2)$-embedding of~$C_g$. To avoid redundant notation, we also 
denote by~$C_g$ the curve $\nu_{2,2}(C_g)$. By construction, 
$C_g$ splits as $C_g = C_1\cup C_2$ where $C_1$ is a conic and $C_2$ is a 
quartic rational normal curve. Notice that $\spn \, C_1 \cong \pr^2$ and $\spn 
\, C_2 \cong \pr^4$. Consider the projection from~$[f]$, denoted by $\rho \colon 
\spn \tth C_g \cong \pr^6 \dashrightarrow \pr^5$. Notice that $\rho|_{\spn\, 
C_1}$ and $\rho|_{\spn\, C_2}$ are isomorphisms, because otherwise $C_1$ or 
$C_2$ will be apolar to $f$ which contradicts Lemma 
\ref{lem:orthogonal_biquadratic}. Set $P$ to be the preimage under~$\rho$ of the 
line $\rho \bigl( \spn \, C_1 \bigr) \cap \rho \bigl( \spn \, C_2 \bigr)$, and 
define the lines $L_1 = P \cap \spn \, C_1$ and $L_2 = P \cap \spn \, C_2$. 

Consider the line~$L_1$: by construction, it passes through~$Q = C_1 \cap 
C_2$, the only singularity of~$C_g$, and intersects the conic~$C_1$ in 
another point~$T$. The line through~$[f]$ and~$T$ is contained in the plane~$P$, 
thus intersects~$L_2$ in a point~$\tilde{T}$. Since $C_2$ is a smooth rational 
quartic, the set
\[
  \bigl\{ \mathrm{trisecant\ planes\ of\ } C_2 \mathrm{\ passing\ through\ } 
\tilde{T} \bigr\}.
\]
corresponds to the variety of sum of powers of a quartic bivariate form 
decomposed into three summands, and by Lemma~\ref{rat} it is 
isomorphic to~$\pr^1$. If we pick the three points of intersection of such a 
trisecant plane with~$C_2$ and we add the point~$T$, we obtain four points 
whose span contains~$[f]$. Thus, we have constructed a~$\pr^1$ of schemes 
of length~$4$ apolar to~$f$ constituted of $3$ points lying on~$C_2$, and one 
common
point lying on~$C_1$.

Consider the line~$L_2$: by construction, it passes through the singularity~$Q$, 
and it intersects the secant variety of~$C_2$ (a cubic threefold) in another 
point~$R$. The line through~$[f]$ and~$R$ is contained in the plane~$P$, thus 
intersects~$L_1$ in a point~$\tilde{R}$. Since $R$ is in the secant variety 
of~$C_2$, there exist two points in~$C_2$ whose span contains~$R$. Moreover, 
there is a pencil of lines in~$\spn \, C_1$ passing through~$\tilde{R}$, which 
defines a pencil of length~$2$ schemes on~$C_1$. So we get a pencil of 
length~$4$ schemes apolar to~$f$, all of them with a length~$2$ scheme on~$C_2$ 
in common. 
\end{proof}

To complete the proof of Lemma~\ref{lem:map22} we apply 
Lemma~\ref{lem:2pencils} to $\ell'\cup Z(\tilde{g})$: only the pencil of 
apolar subschemes with a fixed subscheme of length~$2$ on~$Z(\tilde{g})$ is 
mapped to $\ell$ by $\delta_{2,1}$.
\end{proof}

\begin{lemma}\label{lem:3pencils}
Let $g\in f_{2,1}^\bot$ such that $g = l_1\tth l_2\tth \tilde{l}$ for 
forms~$l_i$ of bidegree~$(1,0)$ and~$\tilde l$ of bidegree~$(0,1)$. Then the  
zero locus~$C_g$ of~$g$ supports three pencils of length~$4$ apolar schemes. 
Two of them are fibers of $\Phi_{2,1}$, while $\Phi_{2,1}$ 
maps the third isomorphically to a line in~$G(2,f_{2,1}^{\bot})$.
\end{lemma}
\begin{proof}
The first part of the proof follows a similar argument as that in 
Lemma~\ref{lem:2pencils}, so 
we only provide a sketch. The $(2,2)$-embedding of~$C_g$ splits into three 
conics $C_1, C_2$ and $\tilde C$, such that $C_1\cap \tilde C = \{Q_1\}$, 
$C_2\cap \tilde C =\{ Q_2\}$ and $C_1\cap C_2 = \emptyset$.

By projecting from~$[f]$ one can prove that there exists a plane~$P$ such that 
$P \cap \spn\, C_1 =\ell_1$, $P \cap \spn (\tilde C\cup C_2 ) = \ell_2$ where 
$\ell_1, \ell_2$ are lines, and $P$ contains the line through $[f]$ and $Q_1$. 
The line $\ell_1$ meets $C_1$ in $Q$ and in another point $R_1$. The line 
through~$R_1$ and~$[f]$ meets~$\ell_2$ in a point~$T$. By a similar argument but 
projecting from $T$, we obtain a point $R_2\in C_2$ and a pencil of pairs of 
points $R_3, R_4 \in \tilde C$ such that $T$ is in~$\spn(\{R_2,R_3,R_4\})$. It 
follows that $[f]$ belongs to~$\spn(\{R_1,R_2,R_3,R_4\})$. In this way we find a 
pencil of apolar schemes with $R_1$ and $R_2$ as fixed points, one on each of 
the curves~$C_1$ and~$C_2$, and a moving part of length~$2$ on~$\tilde C$. 
On~$Z_{2,1}$, the curves~$C_1$ and~$C_2$ are mapped to lines by~$\delta_{2,1}$, 
while $\tilde C$ is mapped to a conic. Their union is a plane section of 
$Z_{2,1}$, and the apolar schemes are all collinear. Since $R_1$ and $R_2$ are 
mapped to the same point~$S$, this pencil of apolar schemes must lie on the 
pencil of lines through~$S$. Therefore the image of this pencil in 
$G(2,f_{2,1}^{\bot})$ is a line. Two other pencils of such schemes with mobile 
parts supported on~$C_1$ and~$C_2$ can be constructed similarly. Each of the 
latter two pencils is mapped to a line in~$Z_{2,1}$, namely the images of~$C_1$ 
and~$C_2$, and is therefore, by Lemma~\ref{lem:map22}, the fiber of the 
morphism~$\Phi_{2,1}$ over a point in~$D_{1,2}$. This concludes the proof.
\end{proof}

\begin{proposition}
\label{prop:biquadratic}
For a general bihomogeneous form~$f$ of bidegree~$(2,2)$ the image of the 
map~$\Phi_{2,1}$ defined in~\eqref{eq:birational_22} is a smooth linear 
section of 
the Grassmannian $G(2, f^{\bot}_{2,1})$.
\end{proposition}

\begin{proof}
Since $\dim \VPS ([f],4) = 3$, $\dim G(2,f_{2,1}^\bot) = 4$ and $\Phi_{2,1}$ is 
birational onto its image, the image is a hypersurface~$U$ 
in~$G(2,f_{2,1}^\bot)$. 

The degree~$3$ component of the Chow group of~$G(2,f_{2,1}^\bot)$ is generated 
freely by one Schubert class $\Sigma_1$, so $[U] = d \tth \Sigma_1$ for 
some~$d$. The intersection of an $\alpha$-plane $\Sigma_{2}$ with $\Sigma_{1}$ 
gives the only class~$\Sigma_{3}$ in degree~$1$ in the Chow group. Hence $[U] 
\cdot \Sigma_{2} = d \tth \Sigma_{3}$. We prove that $d = 1$.

Let us consider the intersection of~$U$ with an $\alpha$-plane $\Sigma_2$. 
Every $\alpha$-plane in ~$G(2,f_{2,1}^\bot)$ is of the form   
$\Sigma_2(g) = \{ \langle g_1, g_2\rangle \subseteq f_{2,1}^\bot \colon g\in 
\langle g_1, g_2\rangle \}$
for some $g\in f^\bot_{2,1}$. On the other hand, such a form~$g$ defines a 
rational curve~$C_g$, and its $(2,2)$-embedding in~$\pr^8$ 
has degree~$6$. Therefore, the intersection of~$U$ with $\Sigma_2(g)$
has preimage under~$\Phi_{2,1}$ given by
\begin{align*}
  \Phi_{2,1}^{-1}\bigl(U \cap \Sigma_2(g)\bigr) & = \bigl\{ [\Gamma] \in 
\VPS([f],4) \, : \, g \in 
I_{\Gamma, (2,1)} \bigr\} \\
  & = \bigl\{ [\Gamma] \in \VPS([f],4) \, : \, \Gamma \subseteq C_g \bigr\}.
\end{align*}

If $C_g$ is smooth, then by Lemma~\ref{rat} we derive 
that $\Phi_{2,1}^{-1}\bigl(U \cap \Sigma_2(g)\bigr) \cong \pr^1$. 

Consider now the case when $C_g$ is not smooth. If $C_g \subseteq \pp$ splits 
into the union of a line and a smooth conic (intersecting in a point), then 
$\nu_{2,2}\left(C_g\right) \subseteq \pr^8$ splits into a conic~$C_1$ and a 
quartic~$C_2$, both rational and smooth. In this case, $\Phi_{2,1}^{-1}\bigl(U 
\cap \Sigma_2(g)\bigr)$ has two irreducible components, both rational and 
smooth, by Lemma~\ref{lem:2pencils}. We claim that these are the only two 
components of maximal dimension of the scheme $\Phi_{2,1}^{-1}\bigl(U \cap 
\Sigma_2(g)\bigr)$. In fact, if there were $[\Gamma] \in \VPS([f],4)$ such that 
$\nu_{2,2}(\Gamma) \subseteq C_1$, then the conic~$C_1$ would be apolar 
to~$f$; this would imply that there is a nonzero element in~$f^{\bot}_{1,0}$, 
contradicting Lemma~\ref{lem:orthogonal_biquadratic}. An 
analogous argument excludes the possibility that $\nu_{2,2}(\Gamma) \subseteq 
C_2$. On the other hand, there is at most one scheme~$\Gamma$ formed by three 
points on~$C_1$ and one point~$B$ on~$C_2$. In fact, by construction $[f] \in 
\spn \, \Gamma$, thus the line through~$[f]$ and~$B$ intersects $\pr^2 = \spn \, 
C_1$, so it is contained in the plane~$P$ from Lemma~\ref{lem:2pencils}. Since 
$P \cap C_2$ is the singular point~$Q$, such line coincides with the line 
through~$[f]$ and~$Q$, and that means that we have at most one scheme~$\Gamma$ 
of this kind.

If $C_g \subseteq \pp$ splits into the union of three lines, then 
$\nu_{2,2}\left(C_g\right) \subseteq \pr^8$ splits into three conics~$C_1$, 
$C_2$ and $\tilde C$. In this case, $\Phi_{2,1}^{-1}\bigl(U \cap 
\Sigma_2(g)\bigr)$ has three irreducible components, all rational and smooth, 
but, by Lemma~\ref{lem:3pencils}, only one of them is not contracted 
by~$\Phi_{2,1}$.

Therefore, as the $\alpha$-planes vary, we obtain a family of smooth and 
rational curves. Hence the only possibility is that $U$ is a linear complex, 
i.e.\ a hyperplane section of~$G(2,f_{2,1}^\bot)$. 

Let $X_{2,1}$ be the image of~$\Phi_{2,1}$, we show that $X_{2,1}$ is smooth. 
Assume by contradiction that $X_{2,1}$ is singular. Then it contains two 
families of planes. In particular, it contains a family of $\alpha$-planes as 
planes in $G\left(2, f^{\bot}_{2,1}\right)$. But any $\alpha$-plane in 
$G\left(2, f^{\bot}_{2,1}\right)$ is of the form~$\Sigma_2(g)$ and 
intersects~$X_{2,1}$ in a curve, so cannot be contained in~$X_{2,1}$. This 
proves the claim.
\end{proof}

\begin{proposition}
\label{prop:all_apolar}
  Every $[\Gamma] \in \VPS([f],4)$ is apolar to~$f$.
\end{proposition}
\begin{proof}
Let $[\Gamma] \in \VPS ([f],4)$, we have to show that $I_{\Gamma} \subseteq 
f^{\bot}$. By Lemma~\ref{lem:map22}, both $I_{\Gamma} \cap f^{\bot}_{2,1}$ and 
$I_{\Gamma} \cap f^{\bot}_{1,2}$ are two-dimensional. If both $I_{\Gamma, 
(2,1)}$ and $I_{\Gamma, (1,2)}$ are two-dimensional, then $I_{\Gamma} \subseteq 
f^{\bot}$. Suppose now that $I_{\Gamma, (2,1)}$ has dimension~$3$: then, by 
Lemma~\ref{lem:pencils_21}, the system of $(2,1)$-curves in~$I_{\Gamma} \cap 
f^{\bot}_{2,1}$ must have a common component, a $(1,0)$-line~$\ell$. Suppose 
first that $\ell$ is a common component of the whole system $I_{\Gamma, (2,1)}$. 
Then the residual $3$-dimensional family of $(1,1)$-curves can have at most one 
point in common. Therefore the line~$\ell$ contains a length~$3$ subscheme 
of~$\Gamma$. The image of~$\ell$ under $\delta_{1,2}$ is a conic, but this 
contradicts the fact that $\delta_{1,2}(\Gamma)$ is collinear. Suppose next that 
$\ell$ is not a common component of the linear system~$I_{\Gamma, (2,1)}$. Then 
$I_{\Gamma, (2,1)} = \left\langle g_1, g_2, g_3 \right\rangle$ with $g_3 \not\in 
I_{\Gamma} \cap f^{\bot}_{2,1}$, and $\ell$ and the zero locus of~$g_3$ 
intersect in a point. Let $g_1 = l \tth \tilde{g}_1$ and $g_2 = l \tth 
\tilde{g}_2$, where $l$ is the linear factor corresponding to $\ell$. Then a 
subscheme of length~$3$ of $\Gamma$ is contained in the zero-locus 
of~$\tilde{g}_1$ and~$\tilde{g}_2$. This forces $\tilde{g}_1$ and $\tilde{g}_2$ 
to have a common component~$\tilde{\ell}$, because otherwise their zero locus 
(being the intersection of two $(1,1)$-curves) would have length at most~$2$. 
Then $\tilde{\ell}$ is either a~$(1,0)$ or a $(0,1)$-line. In the first case, we 
can repeat the previous argument and apply~$\delta_{1,2}$ to~$\Gamma$, obtaining 
a contradiction; in the second case we use~$\delta_{2,1}$. The case 
with~$I_{\Gamma, (1,2)}$ is analogous.
\end{proof}

\begin{corollary}
\label{22smooth} 
For a general bihomogeneous form~$f$ of bidegree~$(2,2)$ the 
variety $\VPS ([f],4)$ is a smooth $3$-fold.
\end{corollary}
\begin{proof}
 First of all, the Hilbert scheme $\Hilb_4 \left( \pp \right)$ itself is smooth 
(see~\cite{Fogarty1968}). Consider next, in $\Hilb_4 \left( \pp \right)$, the 
open subset ${\mathcal U}$ of schemes $\Gamma$ whose ideal $I_{\Gamma} 
\subseteq T$ has codimension~$4$ in bidegree $(2,2)$. Over ${\mathcal U}$ we 
consider the rank~$5$ vector bundle $E_{\mathcal U}$ with fiber over a 
scheme~$[\Gamma]$ the dual of the space $I_{\Gamma,(2,2)} \subseteq T_{2,2}$ of 
$(2,2)$-forms in the ideal of~$\Gamma$. The linear form 
\[
  \phi_{f,(2,2)} \colon T_{2,2}\longrightarrow \C
\]
defines a section on $E_{\mathcal U}$. If $\VPS([f],4) \subseteq {\mathcal U}$, 
then $\VPS([f],4)$ is the $0$-locus of such section by 
Proposition \ref{prop:all_apolar}, since in the case of surfaces all apolar 
schemes are in the closure of smooth apolar schemes.

\begin{lemma}
 Let $f\in S_{2,2}$ be general, then $\VPS([f],4) \subseteq {\mathcal U}$.
\end{lemma}
\begin{proof}
It suffices to prove that for any $[\Gamma]\in \VPS([f],4)$, the space 
$I_{\Gamma,(2,2)}$ has dimension~$5$, or equivalently, the image 
$\nu_{2,2}(\Gamma)$ spans a $\pr^3$. For this, assume that $\nu_{2,2}(\Gamma)$ 
spans a plane $P_\Gamma$. 

If the intersection $P_\Gamma \cap \nu_{2,2} \left( 
\pp \right)$ is finite, $\Gamma$ is either curvilinear or it contains the 
neighborhood of a point. In the latter case, $P_\Gamma$ must be a tangent plane 
to $\nu_{2,2} \left( \pp \right)$, but a tangent plane intersects $\nu_{2,2} 
\left( \pp \right)$ only in a scheme of length~$3$, so this is impossible. If 
$\Gamma$ is curvilinear it is contained in a smooth hyperplane section of 
$\nu_{2,2}(\Gamma)$, an elliptic normal curve of degree~$8$. But on any such 
curve any subscheme of length~$4$ spans a~$\pr^3$, again a contradiction. 

Finally, if $P_\Gamma \cap \nu_{2,2} \left( \pp \right)$ is infinite, it 
contains a curve. But the only plane curves on $\nu_{2,2} \left( \pp \right)$ 
are conics, and they are the intersection of their span with $\nu_{2,2} \left( 
\pp \right)$. So in this case $\Gamma$ is contained in a $(0,1)$-curve or a 
$(1,0)$-curve. If $\Gamma$ is apolar to~$f$, this is impossible, so the lemma 
follows. 
\end{proof}

Taking all $\phi_{f,(2,2)}$ for $f\in S_{2,2}$ gives a linear space of 
sections of $E_{\mathcal U}$ without basepoints on~${\mathcal U}$, 
so for a general~$f$ the $0$-locus $\VPS([f],4)$ is smooth. This
proves Corollary~\ref{22smooth}.
\end{proof}

Theorem~\ref{thm:main}\ref{A} is now equivalent to the following:
\begin{theorem} 
For a general bihomogeneous form~$f$ of bidegree~$(2,2)$, the variety 
$\VPS([f],4)$ is isomorphic to the graph of the birational automorphism on a 
smooth quadric threefold~$Q$ given by the linear system 	of quadrics through a 
rational normal quartic curve in~$Q$.
\end{theorem}
\begin{proof}
We first show that the following rational map is an injective 
morphism:
\[
  \begin{array}{rccc}
    \Xi \colon & \VPS ([f],4) & \dashrightarrow & G\left(2, 
f^{\bot}_{2,1}\right) 
\times G\left(2, f^{\bot}_{1,2}\right) \\
    & [\Gamma] & \mapsto & \bigl( I_{\Gamma, (2,1)}\tth , \tth I_{\Gamma, 
(1,2)} \bigr)
  \end{array}
\]
From Proposition~\ref{prop:all_apolar} all schemes $[\Gamma] \in \VPS([f],4)$ 
are apolar to~$f$, so $I_{\Gamma} \subseteq f^{\bot}$. By 
Remark~\ref{remark:collinear}, both images of~$\Gamma$ under~$\delta_{2,1}$ and 
$\delta_{1,2}$ lie exactly on one line, so 
\[
  \dim I_{\Gamma, (1,2)} \, = \, \dim I_{\Gamma, (2,1)} \, = \, 2.
\]
Hence $\Xi$ is a morphism. 

We now show the injectivity of~$\Xi$. From Lemma~\ref{lem:map22} and the fact 
that $\Gamma$ is apolar to~$f$ we have that the only points where $\Xi^{-1}$ is 
possibly not defined are the images of schemes~$\Gamma$ that contain a subscheme 
of length~$2$ on a $(1,0)$-line, and a subscheme of length~$2$ on a 
$(0,1)$-line. If these two subschemes of $\Gamma$ do not intersect, then the 
union of the two lines is defined by a $(1,1)$-form that must be apolar to~$f$, 
contradicting Lemma~\ref{lem:orthogonal_biquadratic}. If the two subschemes 
intersect, the scheme~$\Gamma$ is mapped to a line in both~$Z_{2,1}$ 
and~$Z_{1,2}$. In this case $\Gamma$ has a subscheme of length~$3$ contained in 
the union of a $(0,1)$-line and a $(1,0)$-line and a residual point that lies in 
the double curve and thus is mapped to the singular curve in both $Z_{2,1}$ and 
$Z_{1,2}$. 

Assume that $\Gamma$ and $\Gamma'$ are two apolar 
schemes of length~$4$ and that $\Xi(\Gamma) = \Xi(\Gamma')$.  
Then both $\Gamma$ and $\Gamma'$ have a subscheme of length $3$ contained in 
a pair of lines $L \cup L'$ that together form a $(1,1)$-curve, and they 
each have a residual point that is mapped to a singular point in both $Z_{2,1}$ 
and $Z_{1,2}$. Then for each line~$L$ and~$L'$ the subschemes of~$\Gamma$ 
and $\Gamma'$ residual to the line must coincide. But both 
schemes must also contain the point of intersection of~$L$ and~$L'$, so the 
two schemes coincide. Hence $\Xi$ is injective.

Since $\VPS([f],4)$ is smooth, to complete the proof it suffices 
to identify the image of~$\Xi$ and show that it is smooth.

Now, the collection of lines $D_{2,1} \subseteq X_{2,1}$ as defined 
in~\eqref{eq:D} is a smooth rational quartic curve. It is normal, otherwise it 
would span a $\pr^3$ and therefore be contained in a special linear complex, 
i.e.\ all lines in~$Z_{2,1}$ would intersect some fixed line, which is ruled out 
by Corollary~\ref{cor:morphism} above. In the planes spanned by two intersecting 
lines in~$Z_{2,1}$ the pencil of lines through the intersection point is a line 
in~$X_{2,1}$. For each double point on~$Z_{2,1}$ we obtain such a line, so they 
form a surface scroll~$V_{2,1} \subseteq X_{2,1}$. By construction, $D_{2,1}$ is 
contained in this scroll and intersects the general line in the scroll in two 
points. So $V_{2,1}$ is also contained in the secant variety of~$D_{2,1}$, a 
cubic hypersurface~$\SD_{2,1}$. Therefore the scroll has degree at most $6$ and 
is contained in the complete intersection $\SD_{2,1}\cap X_{2,1}$. 

To see that 
$V_{2,1}=\SD_{2,1}\cap X_{2,1}$
we compute its degree. This is computed from the 
bidegree~$(d_1,d_2)$ in the Grassmannian. Notice that $V_{2,1}$ parametrizes the 
lines in~$X_{2,1}$ that pass through a singular point in~$Z_{2,1}$. Moreover, 
lines $\ell$ in~$\pr\bigl( \left( f^{\bot}_{2,1} \right)^{*} 
\bigr)$ that pass through a singular point $Q$ 
in~$Z_{2,1}$ such that $[\ell] \in X_{2,1}$, all lie in the plane spanned by the 
two lines contained in~$Z_{2,1}$ passing through~$Q$. The number~$d_1$ counts 
the number of lines in a general plane that belong to~$V_{2,1}$. A general plane 
$P$ contains three singular points of $Z_{2,1}$. For each of them, there is one 
line contained in both $Z_{2,1}$ and~$P$ passing through it, so $d_1=3$. The 
number of lines through a general point that belong to $V_{2,1}$ is~$d_2$. A 
general point lies in three planes that intersect~$Z_{2,1}$ in a conic section, 
hence also in two lines, so $d_2=3$. We conclude that $V_{2,1}$ has degree~$6$, 
and so $V_{2,1}$ is a complete intersection. 

Consider now a Veronese surface $\mathcal{V} \subseteq G\left(2, 
f^{\bot}_{2,1}\right)$ that contains~$D_{2,1}$. The Cremona transformation 
on~$\pr^5$ defined by the quadrics in the ideal of~$\mathcal{V}$ contracts the 
secant variety of~$\mathcal{V}$ to a Veronese surface~$\mathcal{V}'$, while the 
strict transform of~$\mathcal{V}$ is mapped to the secant variety 
of~$\mathcal{V}'$. The Cremona transformation restricts to a birational map
\[ 
  \gamma_{2,1} \colon \, X_{2,1} \dashrightarrow X' \subseteq \pr^4 
\]
where $X' \subseteq \pr^4$ is a smooth quadric $3$-fold. In fact, the 
restriction is defined by the quadrics in the ideal of~$D_{2,1}$ in~$X_{2,1}$.  
This space of quadrics is $5$-dimensional, and the image is a hyperplane 
section~$X'$ of the Pl\"ucker quadric, defined by the quadratic relation between 
the quadrics in the ideal of~$D_{2,1}$ as a curve in~$\pr^4$. 

Consider the 
closure of the graph ${\mathcal Y}\subseteq X_{2,1}\times X' $ of the rational 
map~$\gamma_{2,1}$. The strict transform of~$D_{2,1}$ in~${\mathcal Y}$ is 
mapped to a scroll~$T'$ in~$X'$, the intersection of the secant variety of the 
Veronese surface~$\mathcal{V}'$ with the quadric threefold~$X'$. The strict 
transform in~${\mathcal Y}$ of~$V_{2,1}$ is mapped to a rational normal quartic 
curve~$C'$.

We now compare the map~$\gamma_{2,1}$ with the natural birational 
map~$\rho$ sending~$[\ell]$ to $\Phi_{1,2} \bigl( \Phi_{2,1}^{-1}([\ell])\bigr)$
\[ 
  \rho \colon \, X_{2,1} \dashrightarrow X_{1,2}.
\]
Since $\Phi_{2,1}$ is bijective outside the preimage of the curve~$D_{2,1}$, the 
rational map~$\rho$ is defined outside $D_{2,1}$. On the other hand, $\rho$ is 
not defined anywhere on~$D_{2,1}$. The Picard group of~$X_{2,1}$ is generated by 
the hyperplane bundle, so the map $\rho$ must be defined by a $5$-dimensional 
space of sections in $\HH^0 \! \left({\mathcal I}_{D_{2,1}}(d)\right)$ for 
some~$d$, where ${\mathcal I}_{D_{2,1}}$ is the sheaf of ideals of~$D_{2,1}$ on 
the quadric $3$-fold $ X_{2,1}$. To find the degree~$d$ we consider a general 
curve~$C$ defined by a section in~$f^{\bot}_{2,1}$. On the surface $Z_{2,1} 
\subseteq \pr^3$ the curve~$C$ is mapped to a plane quartic curve with a linear 
pencil of lines that cut the curve in the image of schemes of length $4$ that 
are apolar to~$f$. 
This pencil forms a line in~$X_{2,1}$ that does not intersect~$D_{2,1}$. 
Now, the image~$\overline C$ of the curve~$C$ on~$Z_{1,2}$ has degree~$5$. 
The pencil of apolar schemes of length~$4$ on~$C$ is mapped to schemes that are 
collinear also in~$Z_{1,2}$, so it defines on~$\overline C$ a pencil of 
$4$-secant lines. Assuming $C$ is smooth, any two of these $4$-secant lines are 
disjoint, otherwise $\overline C$ would have a plane section of length~$7$ 
or~$8$, impossible. Therefore the pencil of $4$-secant lines are the lines of 
one family of lines in a smooth quadric surface. This means that the image of 
this pencil of lines in~$X_{1,2}$ is a conic, and hence the degree~$d$ is~$2$. 
Since \[\dim \HH^0 \! \left({\mathcal I}_{D_{2,1}} (2)\right) = 5\]
we may 
conclude that the map~$\rho$ coincides with~$\gamma_{2,1}$. Clearly 
$\gamma_{1,2}$ is the inverse of~$\gamma_{2,1}$, and the graph~${\mathcal Y}$ 
of~$\rho$ is the blowup of~$X_{2,1}$ along the smooth curve~$D_{2,1}$, so 
${\mathcal Y}$ is smooth. There is a map from the graph of~$\gamma_{1,2}$ 
to~$\VPS([f],4)$ which sends a graph point to the ideal generated by the two 
pencils, one in each Grassmannian. Therefore the graph~${\mathcal Y}$ is 
identified with $\Xi \bigl( \VPS([f],4) \bigr)$. The graph is smooth, so 
$\VPS([f],4)$ and the graph~${\mathcal Y}$ are isomorphic.
\end{proof}

\section{Bihomogeneous forms of bidegree $(3,3)$}
\label{bicubic}

Let $f$ be a bihomogeneous form in~$S$ of bidegree~$(3,3)$. 
The Segre-Veronese embedding~\eqref{eq:svembedding} is in this case $\nu_{3,3} 
\colon \pp \hookrightarrow \pr^{15}$. 
By~\eqref{eq:rankf}, we have $\rank(f) = 6$ and ${\rm dim} \VPS([f],6) = 2$.

We identify~$\pr(S_{3,3})$ with a linear subspace \[\pr^{15} \subseteq 
\pr \bigl( \C[z_0,z_1,z_2,z_3]_3 \bigr) = \pr^{19},\] 
where $z_0=x_0y_0$, 
$z_1=x_0y_1$, 
$z_2 = x_1y_0$, $z_3 = x_1y_1$. So we can see $[f] \in \pr^{15}$ as a 
cubic form  $[F] \in \pr^{19}$ in $4$ variables.

\begin{remark}
Given $f \in S_{3,3}$ we can associate to it two orthogonal ideals: first of all 
we have the orthogonal~$f^{\bot} \subseteq \C[t_0,t_1][u_0,u_1]$ that we 
introduced and used in the previous sections; moreover, once we interpret~$f$ as 
a cubic form~$F$, we have also~$F^{\bot} \subseteq \C[v_0,v_1,v_2,v_3]$, a 
homogeneous ideal in a polynomial ring in $4$ variables, that act on 
$\C[z_0,z_1,z_2,z_3]$ by differentiation, namely $v_i(f) = \partial/\partial 
z_j(f)$ for $f\in \C[z_0,z_1,z_2,z_3]$.
\end{remark}

\begin{lemma}
\label{lem:orthogonal_cubic}
For a general bihomogeneous form $f\in S_{3,3}$, 
the orthogonal~$f^\bot$ is generated 
by~$5$ bihomogeneous forms of bidegree~$(2,2)$ in~$T$, together with 
$f^{\bot}_{3,1},f^{\bot}_{1,3},f^{\bot}_{4,0}$ and $f^{\bot}_{0,4}$.
\end{lemma}
\begin{proof}
Let us consider the maps $\phi_{f,(a,b)}$
as we did in Section~\ref{biquadratic}. The kernels of these maps 
are the bihomogeneous components of the orthogonal ideal of~$f$. Since $f$ is a
general form, we may assume that the maps~$\phi_{f,(a,b)}$ have maximal rank, 
i.e.\ are either injective or surjective.
Thus we may assume they are injective when 
\[
  (a,b) \in \bigl\{ (0,0), (0,1), (1,0), (1,2), (2,1), (0,3), (3,0) \bigr\}. 
\]
Then, since $\dim T_{2,2} = 9$ and $\dim 
S_{1,1} = 4$, the map $\phi_{f,(2,2)}$ is surjective and $\dim f_{2,2}^\bot 
= 5$. Similarly, we can see that the dimension of $f_{3,1}^\bot$ and 
$f_{1,3}^\bot$ is also~$5$, and that 
\[\dim 
f_{2,3}^\bot  =  \dim f_{3,2}^\bot  =  10 \quad\text{ and } \quad \dim f_{3,3}^\bot = 15.\]
By an analogous procedure to that in the proof of 
Lemma~\ref{lem:orthogonal_biquadratic}, 
it follows that $f_{2,3}^\bot = T_{0,1} \cdot f_{2,2}^\bot$ and 
$f_{3,2}^\bot = T_{1,0} \cdot f_{2,2}^\bot$, and that $f^\bot$ is generated by
$f_{a,b}^\bot$ for $a,b \leq 3$ together with~$T_{4,0} = 
f_{4,0}^\bot$ and~$T_{0,4} = f_{0,4}^\bot$.

We are left to prove that $f_{3,3}^\bot$ is generated by $f^{\bot}_{2,2}, 
f^{\bot}_{3,1}, $ and $f^{\bot}_{1,3}$. If not, then in particular the 
multiplication map $f^{\bot}_{1,3} \otimes T_{2,0}\longrightarrow f_{3,3}^\bot$, 
is not onto. But then there is a relation $gq - g'q' = 0$, where say $g,g'\in 
f^{\bot}_{1,3}$, while $q,q'\in T_{2,0}$. By unique factorization, $q$ and $q'$ 
must have a common factor, so $gl=g'l'$ for some $l,l'\in T_{1,0}$. By 
assumption, $g,g'$ are independent, so $l,l'$ generate $T_{1,0}$ and $g=g_0l'$ 
and $g'=g_0l$. This is possible only if $g_0\in f^{\bot}_{0,3}$, against our 
assumption. 
\end{proof}

Let $f$ be a general bihomogeneous form of bidegree~$(3,3)$ and let $F \in 
\C[z_0,z_1,z_2,z_3]$ be the cubic associated to~$f$. If $F$ is not a cone, the 
orthogonal~$F^{\bot}$ is generated by~$6$ quadrics. By Sylvester's Pentahedral 
Theorem (see~\cite{Sylvester1904b} and for 
example~\cite{Oeding2013}*{Theorem~3.9} and 
\cite{Landsberg2012}*{Example~12.4.2.3}) the powersum variety 
$\VSumP (F, 5) = \VPowerS_{\pr^3} ([F], 5)$ is just a point corresponding to a 
scheme 
$\Gamma_{0} \subseteq \pr^3$ given by a set of~$5$ points. The ideal 
of~$\Gamma_0$ is generated by $5$ quadrics, so a general quadric apolar to~$F$ 
does not intersect~$\Gamma_0$. In fact we may assume that $\nu_{1,1} \left( \pp 
\right)$ in~$\pr^3$ is defined by a general quadric polynomial orthogonal 
to~$F$, and hence 
\begin{equation*}
  \Gamma_0 \cap \nu_{1,1} \left( \pp \right) = \emptyset. 
\end{equation*}
We consider the closure $H_{3t+1}(\Gamma_0)$ in the Hilbert scheme of twisted 
cubic curves, of the set of curves that contain~$\Gamma_0$. A result of 
Kapranov shows that it is a smooth surface.

\begin{proposition}[\cite{Kapranov1993}*{Theorem~4.3.3}]
\label{H3t+1}
$H_{3t+1}(\Gamma_0)$ is isomorphic to a smooth Del Pezzo surface of degree~$5$, 
i.e.\ isomorphic to the blowup of~$\pr^2$ in $4$ points.
\end{proposition} 

% Let $f$ be a general bihomogeneous form of bidegree~$(3,3)$, and let $\Gamma_0$ 
% be the unique set of~$5$ points in~$\pr^3$ that is apolar to the cubic form~$F$ 
% corresponding to~$f$.
\begin{lemma}
\label{lem:smooth_bicubic}
Let $f$ be a general bihomogeneous form of bidegree~$(3,3)$, and let $\Gamma_0$ 
be the unique set of $5$ points in~$\pr^3$ that is apolar to the cubic form~$F$ 
corresponding to $f$. Then for every smooth apolar $[\Gamma] \in 
\VPS([f],6)$, there 
exists a (possibly reducible) twisted cubic curve $C_\Gamma$ passing 
through~$\Gamma_0$ and~$\Gamma$, in particular $[C_\Gamma] \in 
H_{3t+1}(\Gamma_0)$. 
\end{lemma}
\begin{proof}
Consider $6$ general points on $\nu_{1,1} \left( \pp \right) \subseteq \pr^3$. 
They are the intersection of~$\nu_{1,1} \left( \pp \right)$ with a twisted 
cubic curve. This is a particular case of a classical result
often called Castelnuovo's lemma: through $n+3$ points in~$\pr^n$, no~$n$ of 
which lie in a~$\pr^{n-2}$, there is a unique reduced and connected curve of 
degree~$n$ and arithmetic genus~$0$.

Therefore, if $[\Gamma] \in \VPS([f],6)$ is an apolar scheme constituted of~$6$ 
general points, then $\Gamma \subseteq C_\Gamma \subseteq \pr^3$, where 
$C_\Gamma$ is a twisted cubic. One can also show that $\Gamma_0 \subseteq 
C_{\Gamma}$. In fact, by the apolarity lemma, it follows that $\I_\Gamma 
\subseteq F^\bot$, and since $\I_{C_\Gamma} \subseteq \I_\Gamma$ we get 
$\I_{C_\Gamma} \subseteq F^\bot$.  Under the $3$-uple Veronese embedding 
$C_{\Gamma}$ becomes a rational curve of degree~$9$, and since $C_{\Gamma}$ is 
apolar to~$F$, the point~$[F]$ lies in the span of this degree~$9$ curve. 
Therefore $F$ can be interpreted as a general binary form of degree~$9$, and by 
Lemma \ref{rat} such a binary form has rank~$5$, so $[F]$ lies on the span of 
$5$ points belonging to the degree~$9$ curve. On the other hand, the only scheme 
of $5$ points apolar to~$F$ is $\Gamma_0$, therefore those $5$ points are 
nothing but the image of~$\Gamma_0$ under the $3$-uple Veronese embedding, which 
implies that $C_{\Gamma}$ passes through~$\Gamma_0$.

We consider now the other kinds of smooth apolar schemes 
in~$\VPS([f],6)$. If no plane 
passes through $4$ of the points of~$\Gamma$, then we are in the general 
situation and the previous argument shows that we have a unique (smooth) twisted 
cubic through~$\Gamma$ and~$\Gamma_0$. Suppose that exactly $4$ points 
of~$\Gamma$ lie on a plane, and no three of them are on a line. Then there is a 
pencil of conics passing through those planar points, and a line~$\ell$ through 
the remaining two points; thus there exists a unique conic~$C$ in this pencil 
meeting~$\ell$. We prove that $\Gamma_0$ is contained in $C \cup \ell$, which 
hence is an element in $H_{3t+1}(\Gamma_0) $ (corresponding in the Del Pezzo 
surface to a point lying on one of the~$10$ lines of the surface). Under the 
$3$-uple Veronese embedding, the line $\ell$ is mapped to a twisted cubic~$D_1$, 
and the conic~$C$ is mapped to a rational sextic~$D_2$. By construction, the 
point~$[F]$ lies on the span of~$D_1 \cup D_2$. We denote by~$Q$ the point of 
intersection between $D_1$ and $D_2$. We use the same technique as in the proof 
of Lemma~\ref{lem:2pencils} to construct a scheme of length~$5$ apolar to~$F$. 
Let $E_1 = \spn \, D_1$ and $E_2 = \spn \, D_2$, then $E_1 \cong \pr^3$ and $E_2 
\cong \pr^6$. After projection from the point $[F]$ into $\pr^8$ the two linear 
spaces $E_1$ and $E_2$ will intersect in a line, so there is a unique plane~$P$ 
containing the line~$\overline{[F] Q}$ and intersecting~$E_1$ in a 
line~$\ell_{E_1}$ and $E_2$ in a line~$\ell_{E_2}$. The variety of $3$-secant 
planes to~$D_2$ is a quartic hypersurface in~$E_2$, and a general line 
meeting~$D_2$ intersects it in a unique further point. In particular 
$\ell_{E_2}$ intersects~$D_2$ in~$Q$ and the variety of $3$-secant planes in a 
further point~$T$. Therefore we may assume that there are three points~$p_1, 
p_2$ and $p_3$ in~$D_2$ whose span contains~$T$. Consider now the 
line~$\overline{[F] T}$: since it is contained in~$P$, it meets~$\ell_{E_1}$ in 
one point~$R$. A general point in~$E_1$ lies in a unique secant to~$D_1$, so we 
obtain two points $p_4, p_5$ in~$D_2$ whose span contains~$R$. In this way $[F] 
\in \spn \bigl( \{p_1, \dotsc, p_5\} \bigr)$. As above $\{p_1, \dotsc, p_5\} = 
\Gamma_0$ under the $3$-uple Veronese embedding, so the lemma follows. 

Eventually, we rule out all the cases that are left. Suppose that 
$3$ of the $6$ points of~$\Gamma$ are collinear on a line~$\ell$; then those 
$3$ collinear points may be replaced in~$\Gamma$ by a scheme of 
length~$2$, so that $f$ is apolar to a scheme of length~$5$ on~$\pp$. If $5$ 
of the~$6$ points lie in a plane then, the five coplanar points 
may be replaced in~$\Gamma$ by a scheme of length~$4$, so that $f$ is apolar to 
a scheme of length~$5$ on~$\pp$. In both cases this is against the generality 
assumption of~$f$.
\end{proof}

We now reformulate and prove Theorem~\ref{thm:main}\ref{B}.

\begin{theorem}
\label{bicubictheorem}
For a general bihomogeneous form~$f$ of bidegree~$(3,3)$ the 
variety~$\VPS([f],6)$ is isomorphic to a smooth Del Pezzo surface of degree~$5$.
\end{theorem}
\begin{proof}
Let $\Gamma_0$ be the set of $5$ points apolar to the cubic form $F$ associated 
to~$f$ as in Lemma~\ref{lem:smooth_bicubic}.
Let $H_{3t+1}(\Gamma_0)$ be the Hilbert scheme of twisted cubic curves 
through~$\Gamma_0$. 

If $[C]\in H_{3t+1}(\Gamma_0)$, then $C$ is a cubic curve through~$\Gamma_0$, 
that is apolar to~$F$. Moreover $\nu_{1,1} \left( \pp \right) \cap C$ is a 
scheme of length~$6$. In fact every component of $C$ contains some subset 
of~$\Gamma_0$ and therefore intersects~$\nu_{1,1} \left( \pp \right)$ properly.
Thus, we get a morphism
\[
  \psi \colon  H_{3t+1}(\Gamma_0) \longrightarrow \VPS([f],6)
\]
This morphism is injective, because otherwise there would be two cubic 
curves~$C$ and~$C'$ that pass through~$\Gamma_0$ and have a common intersection 
with $\nu_{1,1}\left( \pp \right)$. Since $\Gamma_0$ has no common point with 
$\nu_{1,1} \left( \pp \right)$ this is impossible by Castelnuovo's lemma. To 
show that the morphism~$\psi$ is surjective, we first note that both 
$H_{3t+1}(\Gamma_0)$ and the variety $\VPS([f],6)$ are surfaces, so it suffices 
to show that $\psi$ is onto the set of smooth schemes in~$\VPS([f],6)$. This is 
precisely the content of Lemma~\ref{lem:smooth_bicubic}.

It remains to show that the bijective morphism $\psi$ is an isomorphism. 

\begin{lemma}
\label{all6apolar} 
If $f$ is a general $(3,3)$-form and $[\Gamma]\in \VPS([f],6)$ 
then $\Gamma$ is apolar to~$f$.
\end{lemma}
\begin{proof}
The ideal of each curve~$C$ in~$H_{3t+1}(\Gamma_0)$ is contained in the ideal 
of~$\Gamma_0$ and is therefore apolar to the cubic form~$F$ associated to~$f$. 
The scheme of intersection $\nu_{1,1} \left( \pp \right) \cap C$ is therefore 
apolar to~$f$. This intersection has length~$6$ and belongs to the closure of 
the smooth apolar schemes in~$\VPS([f],6)$. Since $\psi$ is a surjective 
morphism, the lemma follows. 
\end{proof}

To show that $\psi$ is an isomorphism, we show that $\VPS([f],6)$ is smooth. 
First of all, the Hilbert scheme $\Hilb_6 \left( \pp \right)$ itself is smooth 
(see~\cite{Fogarty1968}). Consider next, in~$\Hilb_6 \left( \pp \right)$, the 
open subset~${\mathcal U}$ of schemes~$[\Gamma]$ that lie on a unique curve in 
the Hilbert scheme of twisted cubic curves in~$\pr^3$, and whose ideal on~$\pp$ 
has $\dim I_{\Gamma,(3,3)}=10$, or equivalently, such that the span 
of~$\nu_{3,3}(\Gamma)$ is a~$\pr^5$. Over ${\mathcal U}$ we consider the rank 
$10$ vector bundle~$E_{\mathcal U}$ whose fiber over a scheme~$[\Gamma]$ is the 
dual of the space $I_{\Gamma,(3,3)} \subseteq T_{3,3}$ of $(3,3)$-forms in the 
ideal of~$\Gamma$. The linear form $\phi_{f,(3,3)} \colon T_{3,3}\longrightarrow 
\C$ defines a section on~$E_{\mathcal U}$.  If $\VPS([f],6) \subseteq {\mathcal 
U}$, then $\VPS([f],6)$ is the $0$-locus of this section, by 
Lemma~\ref{all6apolar}.

\begin{lemma}
If $f\in S_{3,3}$ is general, then $\VPS([f],6)\subseteq {\mathcal U}$.
\end{lemma} 
\begin{proof}
It suffices to prove that for any $[\Gamma]\in \VPS([f],6)$ the span 
of~$\nu_{3,3}(\Gamma)$ is a~$\pr^5$.  For a general~$f$, consider the $5$ 
points~$\Gamma_0$ apolar to the cubic form~$F$ on~$\pr^3$ associated to $f$.   
We may assume that every line through a pair of points of~$\Gamma_0$ 
intersects~$\nu_{1,1} \left( \pp \right)$ transversally. Therefore the 
intersection of any cubic curve in $H_{3t+1}(\Gamma_0)$ with $\nu_{1,1} \left( 
\pp \right)$ is curvilinear. If the cubic curve has a component of degree~$d$, 
then the intersection with $\nu_{1,1} \left( \pp \right)$ has degree~$2d$. On 
the $3$-uple embedding of this curve, any such curvilinear scheme spans 
a~$\pr^5$.
\end{proof}
 
Taking all $\phi_{f,(3,3)}$ for $f\in S_{3,3}$ gives a linear space of 
sections of~$E_{\mathcal U}$ without basepoints on~${\mathcal U}$, 
so for a general~$f$ the $0$-locus $\VPS([f],6)$ is smooth.
Since $\psi$ is a bijective map between smooth surfaces, it is an isomorphism.
\end{proof}

\section{Cubic forms on a cubic surface scroll}
\label{cubicone}

Let $\Sigma$ be a cubic scroll in~$\pr^4$. 
The Picard group $\Pic(\Sigma)$ is free of rank~$2$ generated by the class of 
curves $E$ and $F$, where $E^2=-1,\, E \cdot F=1$ and $F^2=0$, see for instance 
\cite{Beauville}*{Proposition IV.1} or \cite{Friedman}*{Chap.~5, Lemma~2.}. 
The linear system $|E+F|$ defines a morphism $\pi \colon \Sigma \longrightarrow 
\pr^2$, which is the blowup of a point $p_E\in \pr^2$ with exceptional divisor 
$\pi^{-1}(p_E)=E$. The Cox ring of $\Sigma$ is isomorphic to a 
bihomogeneous polynomial ring $T = \C[t_0,t_1,u_0,u_1]$ such that 
\begin{align*}
  T_E & = \HH^0\bigl(\Sigma, {\mathcal O}_\Sigma(E)\bigr) = \langle 
t_0\rangle, 
\\
  T_F & = \HH^0\bigl(\Sigma, {\mathcal O}_\Sigma(F)\bigr) = \langle u_0, 
u_1\rangle, \\
  T_{E+F} & = \HH^0\bigl(\Sigma, {\mathcal O}_\Sigma(E+F)\bigr) = \langle 
t_0 u_0, t_0 u_1, t_1\rangle.
\end{align*}
 Let $S = \C[x_0,x_1,y_0,y_1]$ with $t_0,t_1$ dual to $x_0,x_1$ and $u_0,u_1$ 
dual to $y_0,y_1$, generating an action of~$T$ on~$S$ by differentiation, that 
defines the apolarity of the introduction in coordinates.
In fact, we may then interpret $\Sigma \subseteq \pr(S_{E+2F})$ as a set of 
forms:
\begin{multline*}
  \Sigma = \Bigl\{ \bigl[a_0 \tth x_0 \tth l(y_0,y_1) + a_1 \tth x_1 \tth 
l(y_0,y_1)^2\bigr] \in \,\pr(S_{E+2F}) \, \colon \\ a_0, a_1\in \C, \, 
l(y_0,y_1) 
\in \left\langle y_0, y_1 \right\rangle 
\Bigr\}.
\end{multline*}
Let $f\in S_{3E+6F} \subseteq \mathrm{Sym}^3 S_{E+2F}$. Thus $f$ may be 
interpreted as a cubic form~$G$ on~$\pr^4$ restricted to~$\Sigma$. 
According to the definition of variety of apolar schemes of 
Section~\ref{intro}, we have
\[
 \VPSS([f],8) = \overline{ \Bigl\{[\Gamma]\in \Hilb_8 (\Sigma) \, \colon \, 
[f]\in \spn \, \bigl(\nu_{3E+6F}(\Gamma)\bigr), \; \Gamma \ \text{smooth} 
\Bigr\}},
\]
where $\nu_{3E+6F}$ is the morphism associated to the divisor $3E + 6F$. 

The following theorem is equivalent to Theorem~\ref{thm:main}\ref{C}. 

\begin{theorem} 
For a general $f\in S_{3E+6F}$, the variety $\VPSS([f],8)$ is isomorphic 
to~$\pr^2$ blown up in $8$ points. 
\end{theorem}
\begin{proof}
Recall that we may interpret~$[f]$ as a point~$[G]$ in the linear span inside 
$\pr^{34}= \pr(\C[z_0,z_1,z_2,z_3,z_4]_3)$ of the $3$-uple embedding 
of~$\Sigma$. We may clearly interpret $G$ as a general cubic form in~$\pr^4$. 
Therefore $G$, and hence $f$, is not apolar to any rational quartic curve. In 
particular, we may assume that $I_{f,2E+2F}=I_{f,E+3F}=0$. Furthermore, we may 
assume that
\[
  \phi_{f,2E+3F} \colon T_{2E+3F} \longrightarrow S_{E+3F}
\]
has maximal rank, so $I_{f,2E+3F}= \ker \phi_{f,2E+3F}$ is $2$-dimensional, 
i.e.\ defines a pencil of curves $K \subseteq |2E+3F|$. Notice that, by the 
apolarity lemma~\ref{apolarity}, every curve in~$K$ is apolar to~$f$.

\begin{lemma} 
 For a general $f\in S_{3E+6F}$, the singular curves in~$K$ are 
irreducible nodal curves and the basepoints $\Gamma_0$ of~$K$ are $8$ general 
points in~$\Sigma$.
\end{lemma}
\begin{proof}
Let $\Gamma_0 \subseteq \Sigma$ be $8$ general points. In degree $2E+3F$, the 
ideal of~$\Gamma_0$ is $2$-dimensional.
Furthermore, the set of forms $f'\in S_{3E+6F}$ for which $\Gamma_0$ is apolar 
is the $7$-dimensional subspace in $S_{3E+6F}$ orthogonal to 
$I_{\Gamma_0,{3E+6F}} \subseteq T_{3E+6F}$. These $f'$ are precisely the forms 
that are apolar to every curve in~$K$.
Now, $\pr(S_{3E+6F})$ has dimension~$21$, while the set of pencils 
in~$\pr(T_{2E+3F})$ has dimension~$14$, so the general pencil is apolar to some 
form~$f$ and the lemma follows.
\end{proof}

Any scheme~$[\Gamma]$ in~$\VPSS ([f],8)$ has length~$8$, so it lies on a curve 
in $|2E+3F|$. Therefore, if $\Gamma$ is apolar, it lies on a curve in~$K$. Now 
the base scheme~$\Gamma_0$ of~$K$ has length~$8$, so this scheme is the only one 
of length~$8$ that lies on all curves in~$K$. The other schemes $\Gamma 
\subseteq \Sigma$ of length~$8$ that are apolar to~$f$ lie each on a unique 
curve $C \in K$.

Let $C \in K$. Then $C$ is apolar to~$f$, so we may consider the variety
\[
  \VPSC([f],8) \subseteq \VPSS([f],8).
\]
Let $[\Gamma]\in \VPSC([f],8)$. Then $\Gamma \subseteq C$ is a subset of the 
intersection of $C$ with a curve~$C'$ in $|3E+3F|$. The residual 
part of the intersection $C \cap C'$ is a 
unique point on~$C$ that we denote by~$p_\Gamma$. We thus get a map for 
every $C \in K$:
\[
  \psi_{C} \colon \VPSC([f],8) \longrightarrow C; \qquad [\Gamma] \mapsto 
p_\Gamma.
\]
The map $\psi_{C}$ is defined also on $\Gamma_0$ since any curve $C'+E$, with 
$C'\in K$, lies in $|3E+3F|$ and intersect~$C$ in~$\Gamma_0$ and in the 
residual point $E\cap C$.

Composing $\psi_C$ with the blowup map~$\pi$, we get a morphism
\[
\pi\circ \psi_C \colon \VPSC([f],8) \longrightarrow \pr^2
\]
that we want to extend to all of $\VPSS([f],8)$.
For this, consider, in the Hilbert scheme of length $8$ subschemes of $\Sigma$, 
the open set ${\mathcal U}$ of schemes $\Gamma$ that are contained in a unique 
pencil of curves~$N_\Gamma$ in $|3E+3F|$. Let $\overline\Gamma \subseteq 
\Sigma$ be the baselocus of~$N_\Gamma$. If $\overline\Gamma$ is finite, then it 
has length~$9$ and there is a unique point $p_\Gamma\in \Sigma$ residual 
to~$\Gamma$ in~$\overline\Gamma$. Composing with $\pi$ we get a rational map 
$\psi \colon {\mathcal U} \dasharrow \pr^2$.
Clearly the restriction of~$\psi$ to $\VPSC([f],8)$ extends to the 
morphism~$\psi_{C}$ for every curve $C\in K$. Since $\psi_{C}(\Gamma_0)=\pi(E)$ 
for each~$C$, and every other~$\Gamma$ in~$\VPSS([f],8)$ lies in a unique~$C$, 
we see that the restriction of~$\psi$ to~$\VPSS([f],8)$ extends to a morphism 
\[
  \psi_f \colon \VPSS([f],8)\longrightarrow \pr^2
\]
such that the restriction of~$\psi_f$ to $\VPSC([f],8)$ coincides with~$\psi_C$ 
for each $C\in K$.

We proceed to show that $\psi_C$ is an isomorphism for every curve $C \in K$.
For this we first give a more general fact for elliptic curves, equivalent to 
Lemma~\ref{elliptic}.
\begin{lemma} 
\label{ell} 
Let $C \subseteq \pr^{2d-2}$ be an elliptic normal curve of degree~$2d-1$, then 
the $d$-secants $\pr^{d-1}$'s to~$C$ that pass through a 
general point in~$\pr^{2d-2}$ correspond one to one to points on~$C$.
\end{lemma}
\begin{proof}
Let $C'\subseteq \pr^{2d-1}$ be an elliptic normal curve of degree~$2d$ 
embedded by 
a line bundle~$\mathcal{L}$, then the $(d-1)$-secant variety of~$C'$ is a 
complete intersection of a pencil of determinantal hypersurfaces of degree~$d$: 
each hypersurface is defined by the minors of a matrix of linear forms
(see~\cite{Fisher}*{Theorem~1.3, Lemma~2.9}, \cite{Room}). 
Furthermore, for a general line that intersects~$C'$ in a point~$q$, every 
point 
outside~$C'$ lies on a unique hypersurface in the pencil, so after 
projecting~$C'$ from~$q$ we get a curve $C$ of degree~$2d-1$. Moreover, the 
$d$-secants of $C$ through a general point in $\pr^{2d-2}$ correspond one to 
one to line bundles of degree~$d$ on~$C$, i.e.\ to points on~$C$.
\end{proof}
\begin{lemma} 
Assume $f$ is general, so that the singular curves in~$K$ are 
irreducible and nodal and the basepoints~$\Gamma_0$ of~$K$ are $8$ points 
disjoint from the exceptional curve $E$. Then the morphism $\psi_C \colon 
\VPSC([f],8) \longrightarrow C$ is an isomophism for every $C\in K$, and every 
$[\Gamma] \in\VPSC([f],8)$ is apolar to $f$.
\end{lemma}
\begin{proof} 
Consider the embedding $C\longrightarrow \pr^{14} \subseteq \pr \left (S_{3E+6F} 
\right)$ defined by the linear system $|(3E+6F)_C|$ of divisors on~$C$, namely 
the linear system of curves $|3E+6F|$ restricted to~$C$. It is the composition 
of the embedding defined by $|4E+6F|$ and the projection from the point $E\cap 
C$. We consider Weil and Cartier divisors on~$C$ (if $C$ is smooth they of 
course coincide). While Weil divisors may have multiplicity one at a node~$p_C$ 
of~$C$, any effective Cartier divisor has multiplicity at least two at~$p_C$. 
Any Weil divisor~$\Gamma$ of degree~$8$ on~$C$ is contained in a unique Cartier 
divisor~$\bar\Gamma$ of degree~$9$ defined on~$C$ by the pencil~$D_\Gamma$ of 
curves in $|3E+3F|$ that contain~$\Gamma$. The uniqueness of~$D_\Gamma$ implies 
both that the map $\psi_C \colon [\Gamma] \mapsto p_\Gamma = \bar\Gamma -\Gamma$ 
is well-defined, and that it is injective as soon as there is a unique divisor 
in the linear system in $|(3E+3F)_C-p_\Gamma|$ that is contained 
in~$\VPSC([f],8)$.
Any curve~$G_\Gamma$ in $|4E+6F|$ that is not a multiple of~$C$ and contains the 
Cartier divisor~$\bar\Gamma$, defines on~$C$ a Weil divisor 
$\Gamma'=G_\Gamma\cap C-\Gamma$ of degree~$8$ that contains the 
point~$p_\Gamma$. Thus $\Gamma+\Gamma'$ is a hyperplane section of $C \subseteq 
\pr^{15} \subseteq \pr(S_{4E+6F})$ and we can define a pair of linear systems
\[
  L_\Gamma:=|(4E+6F)_C-\Gamma'|\quad {\rm and}\quad 
L_{\Gamma'}:=|(4E+6F)_C-\Gamma|
\]
like in Lemma~\ref{ell} above for smooth elliptic curves. Since 
\[\Gamma+\Gamma'\equiv (4E+6F)_C \quad \text{ and }\quad \Gamma+p_\Gamma\equiv(3E+3F)_C,\]
 we get 
$\Gamma'-p_\Gamma\equiv(E+3F)_C$, i.e.\ $\Gamma'\equiv(E+3F)_C+p_\Gamma$. Now, 
$|(E+3F)_C+p|=|(E+3F)_C+p'|$ if and only if $|(F)_C+p|=|(F)_C+p'|$, which in 
turn is equivalent to $p=p'$. Therefore the linear system~$L_\Gamma$ is uniquely 
defined by~$p_\Gamma$.

A general point in~$\pr^{15}$ lies in the span of a unique divisor in each of 
these linear systems of degree~$8$. So, after projection from the point $E \cap 
C$, the subschemes~$\Gamma$ of length~$8$ 
on $C$ whose span contains a general point $[f] \in \pr^{14}$ in the span of $C \subseteq 
\pr^{14} \subseteq \pr(S_{3E+6F})$ are in one to one correspondence with linear 
systems~$L_\Gamma$, and hence with the points~$p_\Gamma$ on~$C$. And the 
correspondence coincides with the map $\psi_C \colon \VPSC([f],8) 
\longrightarrow C$ above.
   
Every apolar smooth scheme $[\Gamma] \in \VPSC([f],8)$ spans a~$\pr^7$ under the 
embedding by~$|(3E+6F)_{C}|$. If $[\tilde{\Gamma}] \in \VPSC([f],8)$ is a limit 
point, then there is a~$\pr^7$ containing both~$\tilde{\Gamma}$ and~$[f]$. If 
$\tilde{\Gamma}$ spans such~$\pr^7$, then $\tilde{\Gamma}$ is apolar to~$f$. 
Otherwise $\tilde{\Gamma}$ spans at most a~$\pr^6$, implying that $C$ 
contains a scheme of length~$8$ spanning at most a~$\pr^6$. This is impossible: 
for any $6$ points $P_1, \dotsc, P_6$ on~$C$, the subscheme  $\Delta = 
\tilde{\Gamma} \cup \{ P_1, \dotsc, P_6\}$ would span a~$\pr^{12}$, hence each 
hyperplane through~$\Delta$ would meet~$C$ in another point 
outside~$\tilde{\Gamma}$, so $C$ would be rational, while it is elliptic.
\end{proof}

Every $[\Gamma]\in \VPSS([f],8)$ belongs to $\VPSC([f],8)$ for some $C\in K$, 
so in particular, every $[\Gamma]\in \VPSS([f],8)$ is apolar to~$f$.
Consider therefore the open subset ${\mathcal U'} \subseteq {\mathcal U} 
\subseteq \Hilb_8 (\Sigma)$ of the smooth open set ${\mathcal U}$ above 
consisting of schemes~$\Gamma$, such that $\dim I_{\Gamma,3E+6F}=18$, or 
equivalently, such that $\nu_{3E+6F}(\Gamma)$ spans a~$\pr^7$. Let $E_{\mathcal 
U}$ be the vector bundle of rank~$18$ over~${\mathcal U}$ whose fiber 
over~$[\Gamma]$ is the dual of the space of sections in degree $3E+6F$ of the 
ideal $I_{\Gamma, 3E+6F} \subseteq T_{3E+6F}$. The linear form 
$\phi_{f,3E+6F} \colon T_{3E+6F} \longrightarrow \C$ defines a section on 
$E_{\mathcal U}$. If $\VPSS([f],8) \subseteq {\mathcal U'}$, then 
$\VPSS([f],8)$ is the $0$-locus of this section, since any~$[\Gamma]$ in 
$\VPSS([f],8)$ is apolar to~$f$.
\begin{lemma} 
If $f\in S_{3E+6F}$ is general, then $\VPSS([f],8) \subseteq {\mathcal U'}$.
\end{lemma}
\begin{proof}
It suffices to show that for any $[\Gamma]\in \VPSS([f],8)$ the image 
$\nu_{3E+6F}(\Gamma)$ spans a~$\pr^7$. But this follows from the 
fact that $\Gamma \subseteq C$ for some irreducible curve~$C$ in~$K$, and any 
subscheme of length~$8$ on the curve $\nu_{3E+6F}(C)$ spans a~$\pr^7$.
\end{proof}

Taking all $\phi_{f,3E+6F}$ for $f\in S_{3E+6F}$ gives a linear space of 
sections of~$E_{\mathcal U}$ without basepoints on ${\mathcal U}$, 
so for a general $f$ the $0$-locus $\VPSS([f],8)$ is smooth.

Now, every point outside~$\Gamma_0$ lies in a unique curve $C\in K$, so 
$\psi_f$ 
is a birational morphism from a smooth surface and has an inverse that is 
defined outside~$\pi(\Gamma_0)$. Let $\pi' \colon \Sigma'\longrightarrow \pr^2$ 
be the blowup along~$\pi(\Gamma_0)$. Since, by assumption, all $C\in K$ are 
smooth at~$\Gamma_0$, the inverse map to~$\psi_f$ lifts to a morphism $\psi_f' 
\colon \Sigma'\longrightarrow \VPSS([f],8)$ that restricts to the inverse 
of~$\psi_C$ on the strict transform of~$\pi(C)$ on~$\Sigma'$. Therefore 
$\psi_f'$ is an inverse of~$\psi_f$, and hence an isomorphism.
\end{proof}

\def\cprime{$'$}

\section*{Acknowledgments}
 N.\ Villamizar acknowledges the support of RICAM, Linz, where she developed 
part of the research contained in this paper. M. Gallet would like to thank 
Josef Schicho and Hamid Ahmadinezhad for helpful comments, especially about the 
introduction. M.\ Gallet is supported by Austrian Science Fund (FWF): 
W1214-N15, Project DK9 and (FWF): P26607 and (FWF): P25652. K.\ Ranestad 
acknowledges funding from the Research Council of Norway (RNC grant 239015).

\end{document}